\documentclass[11pt]{amsart}
\usepackage{amssymb,xspace,hyperref,mathrsfs}

\PassOptionsToPackage{hyphens}{url}\usepackage{hyperref}

\makeatletter
\g@addto@macro{\UrlBreaks}{\UrlOrds}
\makeatother
\usepackage{color}
\definecolor{darkred}{rgb}{0.4,0,0}
\definecolor{darkgreen}{rgb}{0,0.5,0}
\definecolor{darkblue}{rgb}{0,0,0.4}
\hypersetup{
	pdftitle={Kummer rigidity for K3 surface automorphisms via Ricci-flat metrics},	
	pdfauthor={S. Filip, V. Tosatti},	
	pdfsubject={Kummer rigidity for K3 surface automorphisms via Ricci-flat metrics},		
	pdfkeywords={K3 surfaces},	
	pdfnewwindow=true,		
	colorlinks=true,		
	linkcolor=darkblue,		
	citecolor=darkred,		
	filecolor=darkblue,		
	urlcolor=darkblue,		
	pdfborder={0 0 0},
	breaklinks=true
}

\newtheoremstyle{mytheoremstyle}
    {5pt}	                
    {5pt}                    	
    {\itshape}                  
    {}                          
    {\bfseries}                 
    {.}                         
    {.5em}                      
    {}  			

\newtheoremstyle{mydefinitionstyle}
    {5pt}	                
    {5pt}                    	
    {}                  	
    {}                          
    {\bfseries}                 
    {.}                         
    {.5em}                      
    {}  			

\theoremstyle{mytheoremstyle}

\swapnumbers

\newtheorem{theorem}{Theorem}[subsection]

\numberwithin{equation}{subsection}

\makeatletter
\let\c@theorem \c@equation
\makeatother

\makeatletter
    \let\c@subsubsection\c@equation
\makeatother

\makeatletter
  \let\c@figure\c@equation		
\makeatother

\usepackage{aliascnt}

\newaliascnt{conjecture}{theorem}

\aliascntresetthe{conjecture}

\newaliascnt{lemma}{theorem}
\newtheorem{lemma}[lemma]{Lemma}
\aliascntresetthe{lemma}

\newaliascnt{proposition}{theorem}
\newtheorem{proposition}[proposition]{Proposition}
\aliascntresetthe{proposition}

\newaliascnt{corollary}{theorem}
\newtheorem{corollary}[corollary]{Corollary}
\aliascntresetthe{corollary}

\theoremstyle{mydefinitionstyle}

\newaliascnt{exercise}{theorem}

\aliascntresetthe{exercise}

\newaliascnt{definition}{theorem}
\newtheorem{definition}[definition]{Definition}
\aliascntresetthe{definition}

\newaliascnt{remark}{theorem}
\newtheorem{remark}[remark]{Remark}
\aliascntresetthe{remark}

\newaliascnt{example}{theorem}

\aliascntresetthe{example}

\newaliascnt{question}{theorem}

\aliascntresetthe{question}

\def\subsek~{\S{}}

\def\equationautorefname~#1\null{%
  Eqn.~(#1)\null
}

\usepackage{etoolbox}
\makeatletter
\patchcmd{\@maketitle}
  {\ifx\@empty\@dedicatory}
  {\ifx\@empty\@date \else {\vskip3ex \centering\footnotesize\@date\par\vskip1ex}\fi
   \ifx\@empty\@dedicatory}
  {}{}
\patchcmd{\@adminfootnotes}
  {\ifx\@empty\@date\else \@footnotetext{\@setdate}\fi}
  {}{}{}
\makeatother

\title[Kummer rigidity for K3 surface automorphisms]{Kummer rigidity for K3 surface automorphisms via Ricci-flat metrics}

\date{August 2018}
\thanks{Revised \textsc{\today}}

\author{Simion Filip}
\address{
\parbox{0.5\textwidth}{
	Department of Mathematics\\
	University of Chicago\\
	5734 S University Ave\\
  Chicago IL, 60637}
}
\email{{sfilip@math.uchicago.edu}}

\author{Valentino Tosatti}
\address{
\parbox{0.5\textwidth}{
	Department of Mathematics\\
	Northwestern University\\
	2033 Sheridan Road,
	Evanston, IL 60208}
}
\email{tosatti@math.northwestern.edu}

\newcommand{\isom}{\cong}
\newcommand \toisom {\tilde{\longrightarrow}}

\newcommand{\Kahler}{{{K}\"ahler}\xspace}
\newcommand{\dVol}{\operatorname{dVol}}
\newcommand{\norm}[1]{\left\|#1\right\|}
\newcommand{\ov}[1]{\overline{#1}}
\newcommand \cW {\mathscr{W}}
\newcommand \bC {\mathbb{C}}
\newcommand \bH {\mathbb{H}}
\newcommand \bR {\mathbb{R}}
\newcommand{\id}{\mathbf{1}}
\DeclareMathOperator{\SL}{{SL}}
\DeclareMathOperator{\SU}{{SU}}
\DeclareMathOperator{\GL}{{GL}}
\DeclareMathOperator{\dist}{{dist}}

\newcommand{\del}{\partial}
\newcommand{\de}{\partial}
\newcommand{\dbar}{\overline{\del}}
\newcommand{\ddbar}{i\del\dbar}

\newcommand{\ve}{\varepsilon}
\newcommand{\vp}{\varphi}
\newcommand{\ti}[1]{\tilde{#1}}

\renewcommand{\leq}{\leqslant}
\renewcommand{\geq}{\geqslant}

\begin{document}

\begin{abstract}
We give an alternative proof of a result of Cantat \& Dupont, showing that any automorphism of a K3 surface with measure of maximal entropy in the Lebesgue class must be a Kummer example.
Our method exploits the existence of Ricci-flat metrics on K3s and also covers the non-projective case.
\end{abstract}

\maketitle


\section{Introduction}
  \label{sec:introduction}

A basic result of Yomdin \cite{Yo}, known previously as the Shub Entropy Conjecture, says that the topological entropy of any smooth map of a compact manifold is bounded below by the spectral radius of the action on homology.
Gromov \cite{Gro} showed that in fact for compact \Kahler manifolds and holomorphic automorphisms, this lower bound is always achieved.
Thus entropy can be computed from linear-algebraic data, and when it is positive the measure of maximal entropy is unique on compact \Kahler surfaces by \cite[Thm.~1.2]{deThelinDinh_DynamicsKahler}.
For an introduction to complex dynamics in higher dimensions see \cite{FS}, \cite{Si}, or \cite{Cantat_SurveyMilnor}.

\emph{Kummer examples} are tori with automorphisms that become affine on the universal cover, together with their modifications using basic operations of birational geometry.
A characteristic feature of Kummer examples is that the measure of maximal entropy is in the Lebesgue class.
In \cite{CD}, Cantat-Dupont proved that automorphisms with positive topological entropy of {\em projective} surfaces with measure of maximal entropy in the Lebesgue class are Kummer examples.

The main goal of this article is to give an alternative proof in the case of K3 surfaces, which in addition covers the non-projective case.

\subsection{Main statements}
\begin{theorem}\label{main}
Let $X$ be a $K3$ surface, $T\colon X\to X$ an automorphism with positive topological entropy whose measure of maximal entropy is absolutely continuous with respect to the Lebesgue measure.
Then $X$ is a Kummer K3 and $T$ is induced by an (affine) automorphism of the corresponding torus.
\end{theorem}

This resolves a conjecture of Cantat \cite[p.162]{Cantat_thesis} and McMullen \cite[Conjecture 3.31]{McM2}, including the case of non-projective $K3$ surfaces.
Combining \cite{CD} with our main theorem we easily obtain the following generalization:

\begin{corollary}\label{coro}
Let $X$ be a compact complex surface, $T\colon X\to X$ an automorphism with positive topological entropy whose measure of maximal entropy is absolutely continuous with respect to the Lebesgue measure. Then $(X,T)$ is a Kummer example.
\end{corollary}

As another application, our main result implies that the measure of maximal entropy in McMullen's construction of a Siegel disc \cite{McM} (which are never projective) cannot be in the Lebesgue class.
Whether the complement of the Siegel disc can have positive Lebesgue measure remains an interesting open question.

For rational maps of $\mathbb P^1$ an analogous result was established by Zdunik \cite{Zd}, and for general endomorphisms of $\mathbb P^n$ by Berteloot--Dupont \cite{BD} and Berteloot--Loeb \cite{BL}, with the role of Kummer examples now played by Latt\`es maps.
For related results in the case of general endomorphisms of K\"ahler manifolds see \cite{Ca3}.
For an introduction to K3 surfaces, including their Ricci-flat metrics, see \cite{K3sem}.


\subsection{Proof Outline}
	\label{ssec:proof_outline}

Let $\mu$ be the measure of maximal entropy and $\dVol$ the normalized volume form induced by the holomorphic $2$-form on the K3 surface.
The assumption says that $\mu=f\dVol$ for some $f\in L^1(\dVol)$.

Because $\dVol$ is invariant under any holomorphic automorphism and since $\mu$ is ergodic (in fact mixing by \cite[Thm. 6.1]{Cantat_K3}), it follows that $f$ is the normalized indicator function of a set of positive Lebesgue measure and hence we can assume

\begin{align}
  \label{eqn:mu_S_dVol}
    \mu = \frac{1}{\dVol(S)}\dVol\vert_S
\end{align}
for some $T$-invariant set $S$ of positive Lebesgue measure.
The proof is then naturally divided into three separate steps:
\begin{itemize}
	\item \emph{Step 1:} We prove the result in the special case $\mu = \dVol$ (i.e. $S=X$ above).
	\item \emph{Step 2:} We prove the result when $\mu$ is uniformly hyperbolic and in the Lebesgue class, by reducing to the previous case.
	\item \emph{Step 3:} We prove that $\mu$ is uniformly hyperbolic, assuming that it is in the Lebesgue class.
\end{itemize}

Step 1 is handled in \autoref{sec:ergodic_case} using the Ricci-flat \Kahler metric provided by Yau's theorem \cite{Ya}.
This key input, combined with Jensen's inequality, allows for an elementary proof.
It already covers the case when the stable/unstable eigencurrents are smooth (since the indicator function of the set $S$ has to be smooth, hence $S=X$), which was conjectured by Cantat \cite[p.162]{Cantat_thesis}.
The two key properties that are used throughout are that the metric is \Kahler (so certain integrals are computed cohomologically) and that the metric induces the same volume form as the holomorphic $2$-form.
Each of these conditions can be easily ensured individually, but their simultaneous validity is Yau's theorem.

Step 2 is handled in \autoref{sec:the_uniformly_hyperbolic_case} and is quite general: it would apply to smooth volume-preserving complex surface diffeomorphisms which are uniformly hyperbolic on an ergodic component of Lebesgue measure.

Finally, Step 3 is handled in \autoref{gen}.
We show that the expansion/con\-trac\-tion coefficients of the dynamics are in fact cohomologous to a constant (in the dynamical sense) and then show that the coboundaries are, in fact, uniformly bounded.
We use that $\mu=\eta_+\wedge \eta_-$ for two positive closed currents $\eta_\pm$ called the stable/unstable eigencurrents.
The coboundary structure follows from comparing the conditional measures of $\mu$ on the stable/unstable foliations and the restriction of $\eta_\pm$ to these foliations.
To prove the $L^\infty$ bound for the coboundaries, we first show that the restriction of the stable/unstable eigencurrent to a stable/unstable manifold is flat.
This is then used to derive a contradiction if the uniform bound does not hold by constructing a Brody curve (in an orbifold $K3$ surface which is obtained by contracting some $(-2)$-curves in $X$) which intersects trivially both the stable and unstable eigencurrents.
Such an entire curve is known not to exist by work of Dinh--Sibony \cite{DS}.
This idea is also a key input in the work of Cantat--Dupont \cite{CD}.

\autoref{sec:alternative_arguments_for_the_existence_of_the_coboundaries} contains an alternative derivation of the coboundary structure, based again on Ricci-flat metrics, some estimates of Birkhoff sums, and ideas along the lines of the Gottschalk--Hedlund theorem.
These arguments give {\em a priori} exponential integrability of the coboundaries, and it is possible that one could prove ergodicity of $\dVol$ using this weaker property, rather than the $L^\infty$ bound used in \autoref{sec:the_uniformly_hyperbolic_case}.

\medskip

{\bf \noindent Orbifolds.}
Because the automorphism potentially (in the Kummer case, always) has periodic curves, these must be contracted by a map $\nu\colon X\to Y$ and many arguments happen on $Y$ instead of $X$.
In the K3 case $Y$ has only orbifold singularities and the needed properties are explained in \autoref{ssec:preliminaries_on_currents_and_orbifolds}.
On a first reading, one can assume that $X=Y$ and skip \autoref{contract}.

\medskip

{\bf \noindent Comparison to the approach of Cantat--Dupont.}
Our use of Ricci-flat metrics on K3 surfaces allows for several simplifications compared to the approach in \cite{CD}.
A key step in showing that the K3 is Kummer is based on the existence of expanded/contracted foliations with the required smoothness properties.
In our case, this follows directly from the equality case in Jensen's inequality.
Moreover, the Ricci-flat metric already provides the flat metric on the torus (see \autoref{ssec:an_alternative_argument}) and leads to an alternative proof that we have a Kummer example.

To establish ergodicity (Step 2), we use a standard tool in dynamics, the Hopf argument (in a quantitative form).
\cite{CD} is based on a topological approach, showing that the set where the foliations are defined is both open and closed.
In our case, we construct the expansion/contraction factors $\alpha_u,\alpha_s$ and then establish their uniform boundedness.
We establish their existence using results of De Th\'elin and Dinh \cite{deThelinDinh_DynamicsKahler}, whereas similar objects are constructed in \cite{CD} using the $\lambda$-lemma, holonomy maps, and renormalization along stable/unstable manifolds.

Finally, the use of Dinh--Sibony's \autoref{entire} and the construction of a Brody curve on which the stable/unstable currents vanish is common to both proofs.
We use it to establish uniform hyperbolicity of the measure of maximal entropy, while \cite{CD} use it to obtain a compactness property of the family of stable/unstable manifolds.
The affine structure on these manifolds, a standard fact in dynamics, also appears in both proofs.
The use of the contraction $\nu\colon X\to Y$ is unavoidable in both proofs.

\medskip

{\bf \noindent Acknowledgments.}
We would like to thank Alex Eskin, Carlos Matheus, Federico Ro\-dri\-guez-Hertz, and Amie Wilkinson for discussions, Mattias Jonsson for useful comments on an earlier draft, and Serge Cantat, Curt McMullen and the referee for extensive feedback that significantly improved the exposition.
This research was partially conducted during the period the first-named author served as a Clay Research Fellow, and during the second-named author's visits to the Center for Mathematical Sciences and Applications at Harvard University and to the Institut Henri Poincar\'e in Paris (supported by a Chaire Poincar\'e funded by the Clay Mathematics Institute) which he would like to thank for the hospitality and support. The first-named author gratefully acknowledges support from the Institute for Advanced Study.
This material is based upon work supported by the National Science Foundation under Grants No. DMS-1638352, DMS-1610278 and DMS-1903147.



\section{Preliminaries}\label{prel}
In this section we collect some preliminaries discussions, mostly well-known, and fix notation for the rest of the paper.
On a first reading, the construction of the orbifold $Y$ in \autoref{contract} and its Ricci-flat metrics can be skipped and one can assume that $X=Y$.
An introduction to dynamics on surfaces can be found in the survey \cite{Cantat_SurveyMilnor}.


\subsection{Preliminaries on currents and orbifolds}
  \label{ssec:preliminaries_on_currents_and_orbifolds}

\subsubsection{The eigencurents}
  \label{sssec:the_eigencurents}
For a detailed discussion of the results quoted in this section, see \cite{Cantat_K3}.

Let $T$ be an automorphism with positive topological entropy $h>0$ of a complex $K3$ surface $X$. Then there exist two closed positive $(1,1)$-currents $\eta_\pm$ on $X$ which satisfy
\begin{equation}
T^*\eta_{\pm}=e^{\pm h}\eta_{\pm}
\end{equation}
and have continuous local potentials, in fact H\"older continuous by Dinh--Sibony \cite{DS}.
The currents yield cohomology classes $[\eta_\pm]$ and we normalize them to have $\int_X [\eta_+]\wedge[\eta_-]=1$.
These classes are nef, i.e. are limits of K\"ahler cohomology classes, and have vanishing self-intersection $\int_X [\eta_{\pm}]^2=0$.

We will be particularly interested in the cohomology $(1,1)$-class $[\eta_+]+[\eta_-]$ which is also clearly nef and satisfies $\int_X([\eta_+]+[\eta_-])^2=2\int_X[\eta_+]\wedge[\eta_-]=2$.

\subsubsection{The picture in \texorpdfstring{$H^{1,1}$}{H11}}
  \label{sssec:the_picture_in_h_1_1}
The intersection pairing on $H^{1,1}(X,\bR)$ has signature $(1,19)$ and the class $[\eta_+]+[\eta_-]$ sits on one of the two hyperboloids of classes with square $2$.
The hyperboloid is naturally identified with hyperbolic $19$-space and contains a geodesic determined by intersecting with the two-dimensional real space spanned by $[\eta_+],[\eta_-]$.
Cup product has signature $(1,1)$ on this plane, and is negative definite on its orthogonal complement.
The pullback action of the automorphism $T$ acts as an isometry on the complement to the $2$-plane, and as a translation by $h$ along the geodesic that the plane determines.

The K\"ahler classes in $H^{1,1}(X,\bR)$ are those that pair positively against $(-2)$ curves (see \cite[XIII.2]{K3sem}) and the K\"ahler cone contains the geodesic in its closure.
Typically the geodesic is strictly in the interior and the orbifold construction below is unnecessary (this can be assumed on a first reading).
When the geodesic lies in the boundary, one has to contract the $(-2)$ curves that pair to zero against both $[\eta_+],[\eta_-]$ as in \autoref{contract} below.

For more on constructing automorphisms of K3 surfaces using this point of view, see \cite{McMullen_K3_entropy_glue}.

\subsubsection{The associated orbifold}
The following result is well-known, but for the reader's convenience we provide a proof.
\begin{proposition}\label{contract}
 Let $V\subset X$ be the union of all irreducible compact holomorphic curves $C\subset X$ which satisfy $\int_C([\eta_+]+[\eta_-])=0$.
 \begin{enumerate}
   \item[(i)] There are finitely many compact holomorphic curves periodic\footnote{A $T$-periodic curve is defined to satisfy $T(C)=C$ set-theoretically, but not necessarily pointwise.} under $T$, and $V$ is their union.
   \item[(ii)] There exists an orbifold $Y$ and a holomorphic map $\nu\colon X\to Y$, which is an isomorphism away from $V$ and contracts each connected component of $V$ to an orbifold point of $Y$.
   \item[(iii)] There exists a holomorphic automorphism $T_Y$ of $Y$ with the same topological entropy as $T$, such that $\nu\circ T=T_Y\circ\nu$.
 \end{enumerate}
\end{proposition}

From this result it follows that $X=Y$ if and only if $T$ has no periodic curves, or equivalently the class $[\eta_+]+[\eta_-]$ is K\"ahler on $X$.

\begin{proof}
For part (i), it is clear that $T$-periodic curves are contained in $V$, since $\int_C \eta_\pm = \int_{T^{{-N}}C}(T^N)^*\eta_\pm = e^{\pm Nh}\int_{T^{-N}C}\eta_\pm$ for any compact curve $C$ and integer $N$.
A similar calculation show that $V$ is $T$-invariant and the general result in \cite[Theorem 1.1]{CT} implies that $V$ is the union of finitely many curves, proving (i).

For part (ii), since the intersection form is negative definite on the complement of the span of $[\eta_+],[\eta_-]$, it follows that if we write $V=\cup_{i=1}^NC_i$ for the decomposition of $V$ into irreducible components, then the intersection matrix $(C_i\cdot C_j)$ is negative definite. By a theorem of Grauert \cite{Gr} there is then a contraction map $\nu\colon X\to Y$ onto an irreducible normal compact complex surface, which contracts each connected component of $V$ to a point.

Next we claim that each connected component of $V$ is in fact an $ADE$ curve (i.e. its irreducible components are smooth rational curves with selfintersection equal to $-2$ \cite[\S III.2]{BHPV}). Indeed, the adjunction formula gives
$$p_a(C_i)=1+\frac{(K_X\cdot C_i)+(C_i^2)}{2}\leq\frac{1}{2},$$
using of course that $K_X\cong\mathcal{O}_X,$
hence $p_a(C_i)=0$ and so each $C_i$ is a smooth rational curve with $(C_i^2)=-2$, as claimed.

The fact that $V$ is composed of $ADE$ curves now implies that  all singular points of $Y$ are rational double points, which in particular are orbifold points (locally isomorphic to the quotient $\mathbb{C}^2/\Gamma$ for certain finite subgroups $\Gamma\subset \SU(2)$ acting freely on the unit sphere, see \cite{Du} for more).

For part (iii), following e.g. \cite[Lemma 2.2]{Ka} since $T$ maps $V$ onto itself, it descends to an automorphism $T_Y:Y\to Y$ with $\nu\circ T=T_Y\circ\nu$, and the topological entropy of $T_Y$ equals the one of $T$, namely $h$.
\end{proof}

\subsubsection{The invariant measures}
  \label{sssec:the_invariant_measures}
Because the currents $\eta_\pm$ have continuous local potentials, their wedge product in the sense of Bedford-Taylor \cite{BT} is well-defined and gives a $T$-invariant probability measure $\mu:=\eta_+\wedge\eta_-$ on $X$.
It is the unique measure of maximal entropy.

On the other hand, the K3 surface $X$ carries a nowhere vanishing holomorphic $2$-form $\Omega$, which we normalize to have $\int_X\Omega\wedge\ov{\Omega}=1$.
The probability measure $\dVol:=\Omega\wedge\ov{\Omega}$ is automatically $T$-invariant, and will be referred to as the Lebesgue measure of $X$.
It will sometimes be denoted by $\dVol_X$.

Applying $\nu_*$ gives a $T_Y$-invariant volume form $\dVol_Y:=\nu_*(\dVol_X)$, which is nowhere vanishing in the orbifold sense.

\subsubsection{Orbifold Ricci-flat metrics}
  \label{sssec:orbifold_ricci_flat_metrics}
The orbifold $K3$ surface $Y$ constructed in \autoref{contract} admits orbifold K\"ahler metrics, as is well-known (see e.g. \cite[Ex. 3.2]{Cam}).
In fact, for every $t\in \bR$ there is an orbifold K\"ahler class $[\omega_{Y,t}]$ on $Y$ such that
\begin{align}
  \label{eqn:pullback_from_orbifold}
  e^{t}[\eta_+]+e^{-t}[\eta_-]=\nu^*[\omega_{Y,t}]
\end{align}
holds.
The (nontrivial) proof of this fact is given in \cite[Proof of Theorem 1.3]{FT}.
In our setting, Ricci-flatness of a (normalized) K\"ahler metric $\omega_Y$ is equivalent to $\omega_Y^2 = \dVol_Y$.

\begin{proposition}\label{prop:kahler}
With the notation as in \autoref{eqn:pullback_from_orbifold}, we have:
\begin{enumerate}
  \item[(i)] The K\"ahler class $[\omega_{Y,t}]$ on the orbifold $Y$ contains a unique orbifold Ricci-flat K\"ahler metric $\omega_{Y,t}$.
  \item [(ii)] The metrics satisfy $\omega_{Y,t+h}:=T_Y^*(\omega_{Y,t})$.
\end{enumerate}
\end{proposition}
\begin{proof}
For (i), the proof of Yau's theorem \cite{Ya} extends to orbifolds (see e.g. \cite[Thm. 4.1]{Cam}), and so we conclude that the class $[\omega_{Y,t}]$ contains a unique orbifold Ricci-flat K\"ahler metric $\omega_{Y,t}$.

Part (ii) follows by uniqueness of the Ricci-flat metric in its cohomology class: both $\omega_{Y,t+h}$ and $T_Y^*(\omega_{Y,t})$ are Ricci-flat and we have
\[\begin{split}
\nu^*[T_Y^*(\omega_{Y,t})]&=T^*\nu^*[\omega_{Y,t}]=T^*(e^{t}[\eta_+]+e^{-t}[\eta_-])=e^{t+h}[\eta_+]+e^{-t-h}[\eta_-]\\
&=\nu^*[\omega_{Y,t+h}],
\end{split}\]
and the map $\nu^*$ is injective in cohomology.
\end{proof}



\subsection{General remarks about Lyapunov exponents}
	\label{ssec:general_remarks_about_lyapunov_exponents}
Let $Z$ denote either $X$ or $Y$, and $T_Z$ the corresponding automorphism.
Let $m$ be any $T_Z$-invariant ergodic probability measure.
Fix a smooth hermitian metric for computing all norms below and recall that the Lyapunov exponent of $m$ is defined by (see \cite[\S1.1]{Ledr}):
\[
	\lambda(m) := \lim_{N\to \infty} \frac{1}{N} I_N \text{ with } I_N:=\int_Z \log \norm{DT_Z^N}\, dm
\]
The limit exists since $I_n$ satisfies the subadditivity property $ I_{k+\ell}\leq I_k + I_\ell $:
\begin{align*}
	\int_Z \log \norm{D_xT^{k+\ell}}\, dm(x) & \leq \int_Z \log \left( \norm{D_xT^k}\cdot \norm{D_{T^k x}T^\ell} \right) dm(x)\\
	& = \int_Z \log \norm{D_xT^k} dm(x) + \int_Z \log \norm{D_{T^kx}T^\ell}dm(x)\\
	& = I_k + \int_Z \log \norm{D_xT^\ell}dm(T^{-k}x) = I_k + I_\ell
\end{align*}
where we have used the $T_Z$-invariance of $m$ and the inequality $\norm{A\cdot B}\leq \norm{A}\cdot \norm{B}$ for linear maps $A,B$.
In fact by Fekete's lemma $\lambda(m)=\inf_N \frac{1}{N}I_N$ and in particular
\begin{align}
	\label{eqn:fekete_lower_bound}
	\frac{1}{N}I_N\geq \lambda(m) \quad \forall N \geq 1.
\end{align}
The exponent does not depend on the fixed ambient metric, since any two will be uniformly comparable.

\subsubsection{Stable/Unstable directions}
	\label{sssec:stable_unstable_directions}
If $m$ is a $T_Z$-invariant ergodic probability measure with strictly positive Lyapunov exponent, then the Oseledets theorem (see e.g. \cite[Thm. 3.1]{Ledr}) implies that there exist measurable $DT_Z$-invariant complex line subbundles $W^\pm(x)$ of the tangent bundle of $Y$, defined for $m$-a.e. $x$, such that $\lim_{N\to \infty}\frac{1}{N}\log \norm{D_xT^N_Z\vert_{W^\pm(x)}} = \pm \lambda(m)$ for $m$-a.e. $x$ (the positive and negative exponents have the same absolute value because $T_Z$ is volume-preserving).
We will use alternatively the notation $W^s$ for $W^-$ and call it the stable direction, and $W^u$ for $W^+$ for the unstable.

\subsubsection{Absolute continuity}
  \label{sssec:absolute_continuity}
Suppose now that the measure of maximal entropy $\mu$ is absolutely continuous with respect to the Lebesgue measure.
Then Ledrappier--Young \cite[Corollary G]{LY2} implies that the Lyapunov exponent of $\mu$ is $\frac{h}{2}$, since the real dimension of the unstable subspace is $2$ and the entropy is $h$.\\


\section{The ergodic case}
  \label{sec:ergodic_case}

In this section we assume that in fact $\mu=\dVol$ and give an easy proof that $(X,T)$ is a Kummer example, using Ricci-flat metrics and the equality case of Jensen's inequality.

\subsection{Using Jensen's inequality}

For the orbifold Ricci-flat metrics on $Y$ constructed in \autoref{prop:kahler}, let $\omega_Y:=\omega_{Y,0}$ be the fixed reference metric.

\begin{definition}[Expansion factor]
 	\label{def:expansion_factor}
For $x\in Y$ let $\lambda(x,N)$ denote the expansion factor (or pointwise Lyapunov exponent) after $N$ iterates of the map $T_Y$.	
Namely, at a given point $x$ the metrics $\omega_Y$ and {$\omega_{Y,Nh}=(T_Y^N)^*\omega_{Y}$} are of the form, in an appropriate basis:
\begin{align}
	\label{eqn:omega_omega_N}
	\begin{split}
	\omega_Y(x) & = |dz_1|^2 + |dz_2|^2\\
	\omega_{Y,Nh}(x) & = |e^{\lambda(x,N)}dz_1|^2 + |e^{-\lambda(x,N)}dz_2|^2
	\end{split}
\end{align}
{where as is customary the K\"ahler forms are identified with their corresponding metrics.}
More intrinsically, we can define $\lambda(x,N)$ to be equal to $\frac{1}{2}\log$ of the largest eigenvalue of the hermitian form $\omega_{Y,Nh}(x)$ with respect to $\omega_Y(x)$.
 \end{definition}

\begin{remark}
  \label{rmk:distance_between_metrics}
  The next observation regarding expansion factors will be useful in the proof of \autoref{prop:exponent_foliations_ergodic_case} below (see also \autoref{ssec:yau_vs_oseledets_curve}).
  Suppose $V\isom \bC^2$ and is equipped with a volume form.
  Assume that $h_1,h_2$ are two hermitian metrics on $V$, inducing the background volume form.
  Then the expansion factor can be alternatively defined as
  \[
    \dist(h_1,h_2):= \max_{\norm{v}_{h_1}=1}\log \norm{v}_{h_2}
  \]
  Observe that the quantity is in fact symmetric (because of volume compatibility), i.e. $\dist(h_1,h_2)=\dist(h_2,h_1)$, and nonzero if $h_1\neq h_2$.
  It indeed defines a distance function, making the space of hermitian metrics isometric to hyperbolic $3$-space (see \autoref{sssec:hyperbolic_distances}).

  Observe now that if $h_1\neq h_2$ then the maximizing vector in the definition is unique up to scaling by $S^1\subset \bC$, i.e. it determines a unique complex line, call it the maximizing line between $h_1$ and $h_2$.
  Now suppose that $h_1,h_2,h_3$ are three distinct metrics that satisfy
  \[
    \dist(h_1,h_2) + \dist(h_2,h_3) = \dist(h_1,h_3)
  \]
  Then the maximizing line between the metrics is the same.
  Indeed, let $v'$ be a maximizing vector for $\dist(h_1,h_3)$ and observe that
  \begin{align*}
    \dist(h_1,h_2) + \dist(h_2,h_3) & \geq
    \log\left(\frac{\norm{v'}_{h_2}}{\norm{v'}_{h_1}}\right)
    +
    \log\left(\frac{\norm{v'}_{h_3}}{\norm{v'}_{h_2}}\right)
    \\
    & = \dist(h_1,h_3)
  \end{align*}
  The first inequality is an equality if and only if $v'$ is (up to scaling) maximizing for the pairs $h_1,h_2$ and $h_2,h_3$ so the conclusion follows.
\end{remark}

The following simple observation is the key which yields the main theorem in the case when $\mu=\dVol$.
Note that the assumption is on $X$, but the conclusions are on $Y$.
\begin{proposition}
\label{prop:exponent_foliations_ergodic_case}
Assume that $\mu=\dVol$ on $X$.
Then
\begin{enumerate}
	\item[(i)] We have $\lambda(x,N)=\frac{Nh}{2}$ for every $x\in Y$ and every $N$.
	\item[(ii)] There exist at every $x\in Y$ two orthogonal (for $\omega_Y$) tangent directions $W^\pm(x)$ such that $\log \norm{DT_Y\vert_{W^{\pm}(x)}} = \pm \frac{h}{2}$.
	The directions $W^\pm(x)$ vary real-analytically in $x$, are $T_Y$-invariant, and agree with the directions provided by the Oseledets theorem applied to $\mu$.
\end{enumerate}
\end{proposition}
\begin{proof}
For part (i), the cohomological calculation $\int_Y \omega_Y\wedge \omega_{Y,Nh}=e^{Nh}+e^{-Nh}$ and Jensen's inequality give:
\begin{equation}
\label{jen1}
\begin{split}
\log\left( e^{Nh} + e^{-Nh} \right) &=  \log \left(\int_Y \omega_Y \wedge \omega_{Y,Nh}\right)\\
&\geq\int_Y\log\left(\frac{\omega_Y \wedge \omega_{Y,Nh}}{\dVol_Y}\right)\dVol_Y\\
&=\int_Y \log\left( e^{2\lambda(x,N) } + e^{-2\lambda(x,N)} \right) \dVol_Y.
\end{split}\end{equation}

We established in \autoref{eqn:fekete_lower_bound} that
$$\int_Y \lambda(x,N)\dVol_Y\geq \frac{Nh}{2}.$$
Indeed, in the case at hand $\log \norm{D_xT_Y^N} = \lambda(x,N)$ by definition, and the Lyapunov exponent is $h/2$ by \autoref{sssec:absolute_continuity}.

Note that the function $\log(e^x+e^{-x})$ is convex and increasing so we can apply Jensen again.
Recall that $2I_N := \int_Y 2\lambda(x,N)\dVol_Y\geq Nh$ to find:
\begin{equation}\label{jen2}
	\int_Y \log\left( e^{2\lambda(x,N) } + e^{-2\lambda(x,N)} \right) \dVol_Y \geq \log(e^{2I_N} + e^{-{2I_N}}) \geq \log(e^{Nh}+ e^{-Nh})
\end{equation}
So from \autoref{jen1} and \autoref{jen2} it follows that we must have equality pointwise a.e., that is $\lambda(x,N)=Nh/2$ pointwise a.e., for all $N$.
Since the function $\lambda(-,N)$ is continuous, the result holds everywhere on $Y$.

For part (ii), the equality case in Jensen plus the equality case in $I_k+I_\ell\leq I_{k+\ell}$ (see \autoref{ssec:general_remarks_about_lyapunov_exponents}) imply that at every point $x\in Y$, the directions $dz_1,dz_2$ appearing in \autoref{eqn:omega_omega_N} are independent of $N$ and determine the spaces $W^{\pm}(x)$.
Indeed, we can apply \autoref{rmk:distance_between_metrics} since we have equality in the triangle inequality and hence the direction of maximal expansion at each point is independent of the iterate of the dynamics.
The same applies to the inverse dynamics, giving the two directions.
The directions are orthogonal for $\omega_Y$ by the spectral theorem for hermitian matrices.
Moreover, for any real-analytic \Kahler metrics $\alpha,\beta$, the direction of maximal expansion of $\alpha$ relative to $\beta$ varies real-analytically, away from the locus where the direction is not unique (empty in our case).
\end{proof}

We can now finish off the proof of \autoref{main} when $\mu=\dVol$.

\begin{theorem}\label{main2}
Let $X$ be a $K3$ surface, $T\colon X\to X$ an automorphism with positive topological entropy whose measure of maximal entropy $\mu$ equals Lebesgue measure $\dVol$.
Then $X$ is a Kummer K3 and $T$ is induced by an (affine) automorphism of the corresponding torus.
\end{theorem}
\begin{proof}

\autoref{prop:exponent_foliations_ergodic_case}(ii) gives two line subbundles of the tangent bundle of $Y$, invariant and uniformly expanded/contracted by the dynamics.
By Ghys \cite[Proposition 2.2]{Gh} \cite[Proposition 7.1]{CD}, these give two holomorphic foliations on the orbifold $Y$.
These can be pulled back to holomorphic foliations on $X\backslash V$, which automatically extend to $X$ exactly as in \cite[Proof of Corollary 7.6]{CD}), which are preserved by $T$ (alternatively, as Cantat pointed out to us, in the present case the extension of the holomorphic foliations follows from the explicit description of the singular points of $Y$ as quotient $ADE$ singularities).
At this point we can apply either a result of Cantat \cite[Theorem 7.4]{Cantat_K3}, or a later result of Cantat-Favre \cite[Theorem 3.1]{CF} (which only needs one invariant foliation), to conclude that $X$ is a Kummer K3 and $T$ is induced by an automorphism of the corresponding torus.
Note that these results apply in the K\"ahler case (for \cite{CF} one needs to use results of Brunella \cite{Br,Br2}).
\end{proof}


\subsection{An alternative argument}
  \label{ssec:an_alternative_argument}

At the end of the proof of \autoref{main2} above, after obtaining two $T$-invariant holomorphic foliations on $X$, we appealed to the general results of Cantat \cite{Cantat_K3} or Cantat-Favre \cite{CF} to conclude that $(X,T)$ is a Kummer example.

We now explain how to circumvent in our case some of the just cited arguments using the differential geometry of Ricci-flat metrics.
We suppose that on $X$ we have $\mu=\dVol$, and we have applied \autoref{prop:exponent_foliations_ergodic_case}, to obtain two $T_Y$-invariant transverse holomorphic foliations on $Y$, as above.

\begin{proposition}In this setting, the orbifold Ricci-flat K\"ahler metric $\omega_Y$ on $Y$ is in fact flat.
\end{proposition}
\begin{proof}
At any point $x\in Y$ we have the two orbifold Ricci-flat K\"ahler metrics $\omega_Y$ and $\omega_{Y,h}=T_Y^*\omega_Y$, which have the property that the eigenvalues of $\omega_{Y,h}$ with respect to $\omega_Y$ are $e^h$ and $e^{-h}$, and the corresponding eigenvectors (these are eigenvectors of the endomorphism of the tangent space given by composing $\omega_{Y,h}$ with $\omega_Y^{-1}$) span the stable and unstable holomorphic foliations $\mathcal{F}_{\pm}$ respectively. By construction, these eigenvectors are $\omega_Y$-orthogonal, hence so are $\mathcal{F}_+$ and $\mathcal{F}_-$. Near $x$ we can find local holomorphic functions $z_1,z_2$ such that $\ker dz_1=\mathcal{F}_+,\ker dz_2=\mathcal{F}_-$, which implies that $z_1,z_2$ give local holomorphic coordinates near $x$ (on the local orbifold cover if $x$ is singular). Therefore near $x$ we can write
$$\omega_Y=a idz_1\wedge d\ov{z}_1+b idz_2\wedge d\ov{z}_2,$$
where $a,b$ are local smooth positive functions.
But then
$$0=d\omega_Y=\left(\frac{\de a}{\de z_2}dz_2+\frac{\de a}{\de \ov{z}_2}d\ov{z}_2\right)\wedge idz_1\wedge d\ov{z}_1+\left(\frac{\de b}{\de z_1}dz_1+\frac{\de b}{\de \ov{z}_1}d\ov{z}_1\right)\wedge idz_2\wedge d\ov{z}_2,$$
which imply that $a$ is independent of the $z_2,\ov{z}_2$ variables and $b$ is independent of the $z_1,\ov{z}_1$ directions.

From the definition of curvature we have
$$R_{i\ov{j}k\ov{\ell}}=-\de_i\de_{\ov{j}}g_{k\ov{\ell}}+g^{p\ov{q}}\de_i g_{k\ov{q}}\de_{\ov{j}}g_{p\ov{\ell}},$$
where in our coordinates
$$g_{1\ov{2}}=g_{2\ov{1}}=g^{1\ov{2}}=g^{2\ov{1}}=0,\quad g_{1\ov{1}}=a, g_{2\ov{2}}=b,\quad g^{1\ov{1}}=a^{-1},g^{2\ov{2}}=b^{-1},$$
$$\de_2g_{1\ov{1}}=\de_{\ov{2}}g_{1\ov{1}}=0,\quad \de_1g_{2\ov{2}}=\de_{\ov{1}}g_{2\ov{2}}=0.$$
In particular
$$R_{2\ov{2}1\ov{1}}=-\de_2\de_{\ov{2}}a+g^{p\ov{q}}\de_2 g_{1\ov{q}}\de_{\ov{2}}g_{p\ov{1}}=g^{1\ov{1}}\de_2 g_{1\ov{1}}\de_{\ov{2}}g_{1\ov{1}}=0,$$
while Ricci-flatness gives
$$0=R_{1\ov{1}}=g^{p\ov{q}}R_{p\ov{q}1\ov{1}}=g^{1\ov{1}}R_{1\ov{1}1\ov{1}}+g^{2\ov{2}}R_{2\ov{2}1\ov{1}}=a^{-1}R_{1\ov{1}1\ov{1}},$$
giving $R_{1\ov{1}1\ov{1}}=0$, and
$$0=R_{2\ov{2}}=g^{p\ov{q}}R_{p\ov{q}2\ov{2}}=g^{1\ov{1}}R_{1\ov{1}2\ov{2}}+g^{2\ov{2}}R_{2\ov{2}2\ov{2}}=b^{-1}R_{2\ov{2}2\ov{2}},$$
giving $R_{2\ov{2}2\ov{2}}=0$ (using the K\"ahler identities). Next,
$$R_{1\ov{1}1\ov{2}}=-\de_1\de_{\ov{1}}g_{1\ov{2}}+g^{p\ov{q}}\de_1 g_{1\ov{q}}\de_{\ov{1}}g_{p\ov{2}}=0,$$
$$R_{1\ov{2}2\ov{2}}=-\de_1\de_{\ov{2}}g_{2\ov{2}}+g^{p\ov{q}}\de_1 g_{2\ov{q}}\de_{\ov{2}}g_{p\ov{2}}=g^{2\ov{2}}\de_1 g_{2\ov{2}}\de_{\ov{2}}g_{2\ov{2}}=0.$$
Thanks to the K\"ahler identities, we thus obtain that $R_{i\ov{j}k\ov{\ell}}=0$ for all $i,j,k,\ell$, hence $\omega_Y$ is a flat orbifold K\"ahler metric.
\end{proof}

Once we know that $\omega_Y$ is flat, this in turn implies that there is a finite orbifold cover $\pi\colon Z\to Y$ with $Z$ a compact complex $2$-torus (cf. the discussion in \cite[\S 6.4.3]{CZ}). The arguments in \cite{Cantat_K3,CF} can then be used to show that $T_Y$ lifts to an automorphism of $Z$, which is then affine linear, and that the map $\pi:Z\to Y$ is the quotient by an involution, and so $X$ is a Kummer K3 and $T$ is induced by an automorphism of the corresponding torus.




\section{The uniformly hyperbolic case}
  \label{sec:the_uniformly_hyperbolic_case}

After recalling in \autoref{ssec:recollections_from_pesin_theory} the needed facts from Pesin theory, we prove in \autoref{ssec:ergodicity_in_the_uniformly_hyperbolic_case} that if $\mu$ is uniformly hyperbolic and in the Lebesgue class, then it is in fact equal to $\dVol$.


\subsection{Recollections from Pesin theory}
  \label{ssec:recollections_from_pesin_theory}
This section collects some concepts and results from Pesin theory that will be used to prove ergodicity of $\dVol$, under an extra assumption of uniform hyperbolicity.
While Pesin theory is concerned with the non-uniformly hyperbolic setting, its conclusions apply to sets of almost full measure, on which hyperbolicity is uniform.

For the discussion in this subsection, $Z$ can be any compact complex surface, possibly an orbifold, $T_Z\colon Z\to Z$ a holomorphic automorphism, and $m$ a $T_Z$-invariant ergodic probability measure with nonzero Lyapunov exponents in the sense of \autoref{ssec:general_remarks_about_lyapunov_exponents}.

\begin{definition}[Uniform hyperbolicity]
  \label{def:uniform_expansion}
  We say that $m$ is \emph{uniformly hyperbolic} if there exists a constant $C>0$ such that for $m$-a.e. $x$, we have for $N\geq 1$ that
  \begin{align}
    \label{exp1} \log \norm{D_xT^{-N}_Z\vert_{W^u(x)}} & \leq -\frac{1}{C}N + C\\
    \label{exp2} \log \norm{D_xT^N_Z\vert_{W^s(x)}} & \leq -\frac{1}{C}N + C
  \end{align}
  and the angle between $W^u(x)$ and $W^s(x)$ is bounded below by $\frac{1}{C}$, where everything is measured relative to a fixed smooth Riemannian metric.
\end{definition}
In \cite[Def. 2.2.6]{BarreiraPesin_NonuniformHyperbolicity} sets $\Lambda^\ell_{\lambda\mu\ve j}$ are defined to which most considerations in loc.cit. apply.
A $T_Z$-invariant set of points $x$ which satisfy \autoref{def:uniform_expansion} is then contained in such a $\Lambda^\ell_{\lambda\mu\ve j}$ with $j=1$, $\ve=0$ and $\ell,\lambda,\mu$ only depending on $C$.

\subsubsection{Stable and Unstable manifolds}
  \label{sssec:stable_and_unstable_manifolds}
The following discussion is expanded in \cite[\S7-8]{BarreiraPesin_NonuniformHyperbolicity}.
Because the invariant measure $m$ has nonzero Lyapunov exponents, there exist for $m$-a.e. $x$ unique global immersed stable manifolds $\cW^s(x)$ which contain $x$ and with tangent space equal to $W^s(x)$ there.
For distinct points, stable manifolds either coincide or are disjoint.

Since in our case the stable manifolds are complex $1$-dimensional they are parametrized by $\bC$ \cite[Prop. 2.6]{BLS}.
For convenience of notation, we will use the canonical parametrizations given by the complex line $W^u(x)$, namely the holomorphic maps
\[
  \xi^s_x\colon W^s(x) \toisom \cW^s(x)
\]
normalized to have derivative the identity at the basepoints (where $0\mapsto x$).
Because the only holomorphic automorphisms of $\bC$ are affine, we immediately deduce that
\begin{itemize}
  \item For two points $x,y$ on the same stable manifold the composed map $(\xi^s_y)^{-1}\circ \xi^s_x$ is an affine map from $W^s(x)$ to $W^s(y)$.
  \item The maps $\xi^s_{x}$ and $T^{-1}_Z\circ \xi_{T_Z(x)}^s \circ  D_{x}T_Z $ coincide, because they both induce parametrizations of the stable manifolds at $x$ and have the same derivative at the origin.
  Equivalently $T_Z\circ \xi^s_{x}  =  \xi_{T_Z(x)}^s\circ D_xT_Z$.
\end{itemize}
Therefore the stable manifolds carry canonical affine structures and are parametrized equivariantly for the dynamics.
The same discussion applies to unstable manifolds when using $T_Z^{-1}$ instead of $T_Z$, and we will denote their parametrizations by $\xi^u_x$.

\subsubsection{Charts and size}
  \label{sssec:charts_and_size}
To discuss the geometry of stable manifolds, fix finitely many open charts covering $Z=\cup_\alpha U_\alpha$ and view each chart as equipped with its flat Euclidean metric (in the orbifold sense when necessary).
There exists an $\ve>0$ which is a Lebesgue number of this covering, i.e. for any $x\in Z$ there exists a chart $U_\alpha$ such that the ball of radius $\ve$ around $Z$ is contained in $U_\alpha$.
All considerations below will be in these charts and all objects will be considered only in balls of radius at most $\ve/2$, so that the Euclidean and fixed background metric are comparable, up to uniform constants.

Implicit constants occurring below will be called \emph{uniform} if they only depend on the automorphism $T_Z$, the covering fixed above, and a fixed smooth ambient metric.

\subsubsection{Local stable manifolds}
  \label{sssec:local_stable_manifolds}
Recall that the global stable manifolds $\cW^s(x)$ are only immersed and are patched from local stable manifolds $\cW^s_{loc}(x)$.
The geometry of the local stable manifolds is described in \cite[Thm. 7.5.1]{BarreiraPesin_NonuniformHyperbolicity}.
Most importantly, under the uniform hyperbolicity assumptions in \autoref{def:uniform_expansion} the constants (that appear in loc.cit.) $r$ giving the radius, and $D$ giving the H\"older constant of the derivative, depend only on the constant describing the uniform hyperbolicity (by Thm. 7.5.1(5) of loc.cit.).
As a consequence, using the charts from \autoref{sssec:charts_and_size} to map a ball of radius $r$ in $W^{s/u}(x)$ to $Z$, the stable resp. unstable manifolds will be contained in disjoint cones around $W^{s}(x)$ resp. $W^u(x)$, with angle between the cones uniformly bounded below.



\subsection{Ergodicity in the uniformly hyperbolic case}
  \label{ssec:ergodicity_in_the_uniformly_hyperbolic_case}

\begin{theorem}\label{thm:uniformly_hyperbolic_case}
Suppose that the measure of maximal entropy $\mu$ on $X$ is in the Lebesgue class and additionally under the contraction $\nu\colon X\to Y$ from \autoref{ssec:preliminaries_on_currents_and_orbifolds}, the measure $\nu_*\mu$ is uniformly hyperbolic on $Y$ in the sense of \autoref{def:uniform_expansion}.

Then $\mu=\dVol$ on $X$, and so $(X,T)$ is a Kummer example by \autoref{main2}.
\end{theorem}

The statements $\mu=\dVol_X$ and $\nu_*\mu = \dVol_Y$ are equivalent, since the set contracted by $\nu$ has Lebesgue measure zero.
However, even when $(X,T)$ is a Kummer example, $\dVol_X$ is not uniformly hyperbolic for a smooth metric on $X$.

\begin{proof}
Let $S\subset X$ be the set defined in \autoref{eqn:mu_S_dVol}, so that $\mu=\frac{\dVol}{|S|}\bigg|_S$.
All arguments below are on $Y$ so for simplicity of notation let $S$ denote the image of this set under $\nu$.

Assume that the uniform hyperbolicity condition holds for every point of $S$ (otherwise replace $S$ with the intersection of all iterates of the set on which uniform hyperbolicity holds, still a set of full $\mu$-measure).
It suffices to show that there exists a uniform $\ve>0$ such that for Lebesgue-a.e. $x\in S$ the ball $B(x,\ve)$ of radius $\ve$ around $x$ satisfies $|B(x,\ve)\cap S|=|B(x,\ve)|$, i.e. the ball is essentially (up to Lebesgue measure $0$) contained in $S$ (we'll use $|B|$ to denote the volume of $B$ for the canonical volume form).

This last property, in turn, follows from the Hopf argument and the uniformity of hyperbolicity.
Specifically, from Pesin theory there exist stable and unstable manifolds $\cW^s(x),\cW^u(x)$ through $\mu$-a.e. point $x$.
From our uniform hyperbolicity assumption, the sizes of the local stable/unstable manifolds in the sense of \cite[\S8.1]{BarreiraPesin_NonuniformHyperbolicity} are uniformly bounded below for every $x\in S$, and so are their angles (this is \cite[\S8.1.3]{BarreiraPesin_NonuniformHyperbolicity}).
Indeed there exists a single $\ell>0$ such that the Oseledets-regular level sets $\Lambda^\ell_{\bullet}$ that appear in \cite{BarreiraPesin_NonuniformHyperbolicity} are equal to $S$.
Furthermore, the local stable and unstable manifolds depend continuously in the $C^1$-topology for points in $S$, by \cite[\S8.1.4]{BarreiraPesin_NonuniformHyperbolicity} (even better, H\"older continuity holds along the lines in \cite{ABF_HolderContinuity}).

Now we apply the Hopf argument.
From \cite[Lemma 9.1.2]{BarreiraPesin_NonuniformHyperbolicity}, there exists a set $B_S$ of zero Lebesgue measure such that if $y\in S$ and $z,w\in \cW^u(y)\setminus B_S$ then $z,w\in S$.
Moreover, by the absolute continuity of the unstable foliations \cite[Thm. 8.6.8]{BarreiraPesin_NonuniformHyperbolicity}, for $\mu$-a.e. $y\in S$ the set {$B_S\cap \cW^{u}(y)$} has zero Lebesgue measure in $\cW^u(y)$.
Let $S'\subset S$ be the set of $y$ with this property, so that $|S\setminus S'|=0$.

For $y\in S'$ we have that Lebesgue-a.e. $z\in \cW^u(y)$ is in $S$, hence admits stable manifolds $\cW^s(z)$ of size bounded below and which depend continuously on $z$.
Hence we get a continuous injective map $\Delta\times \Delta\to Y$, where $\Delta=\left\lbrace |z|<1 \colon z\in \bC \right\rbrace$, as follows.
The map is defined by $(z,w)\mapsto \xi^s_z(w)$ where the first factor {of $\Delta\times\Delta$} is identified with a disc of size uniformly bounded below in $\cW^u(y)$, and the second factor with a disc in $W^s(z)$ as $z$ varies in $\cW^u(y)$ (the disc $\Delta$ of radius $1$ is rescaled from the bounded below radius $r$ discussed in \autoref{sssec:local_stable_manifolds}).
While the map is defined only for Lebesgue-a.e. $z$, since it is continuous as a map $z\mapsto Hol(\Delta,Y)$ (by \cite[\S8.1.4]{BarreiraPesin_NonuniformHyperbolicity}), it extends to the desired continuous map $\Delta\times \Delta \to Y$.

By the Invariance of Domain Theorem (\cite[Thm. 2B.3]{Hatcher_AlgebraicTopology}), this continuous injective map contains an open set around $y\in Y$.
It remains to check that this open set contains a ball of size uniformly bounded below.
For this, it suffices to check that there exists a uniform $\ve>0$ such that if $(z,w)\in\Delta\times\Delta$, $y'=\xi_z^s(w)$ and $|y-y'|\leq \ve$ then $|z|+|w|\leq \frac{1}{10}$.
Provided this last property, it is clear that every point within $\ve$ of $y$ is in the image, since the set in question is both open (by invariance of domain) and closed by the property that remains to be checked.

Suppose therefore that $y'=\xi_z^s(w)$.
Recall that in Euclidean triangles, by the law of sines if $a$ is the length of one side and $\alpha$ is the opposite angle, then $\frac{a}{\sin \alpha}$ controls any other side.
In the charts described in \autoref{sssec:charts_and_size}, the stable and unstable manifolds at $z$ are contained in cones around $W^s(z)$ and $W^u(z)$ with angle between them uniformly bounded below (see \autoref{sssec:local_stable_manifolds}).
Therefore there exists a uniform constant $A$ such that $A|y-y'|\geq |z|+|w|$, which suffices for our purposes.
\end{proof}



\section{The general case}
	\label{gen}

We complete the proof of the main theorem by establishing Step 3 from the introduction (\autoref{ssec:proof_outline}).

\subsection{Proving uniform hyperbolicity}
Our goal is the following, which thanks to \autoref{main2} (Step 1) and \autoref{thm:uniformly_hyperbolic_case} (Step 2) completes the proof of \autoref{main}:
\begin{theorem}
	\label{main3}
Suppose that the measure of maximal entropy $\mu$ on $X$ is in the Lebesgue class.
Then under the contraction $\nu\colon X\to Y$, the measure $\nu_*\mu$ is uniformly hyperbolic on $Y$.
\end{theorem}

Because sets contracted under $\nu\colon X\to Y$ have zero Lebesgue measure, $\mu$-measurable functions on $X$ are naturally identified with $\nu_*\mu$-measurable ones on $Y$.
We will not distinguish in notation between functions identified in this manner.

\subsubsection{Expansion/Contraction factors}
	\label{sssec:expansion_contraction_factors}

Recall (see \autoref{sssec:stable_unstable_directions} and \autoref{sssec:absolute_continuity}) that $\nu_*\mu$-a.e. $x\in Y$ has a decomposition of the tangent space:
\[
	\mathbf{T}_x Y = W^s(x) \oplus W^u(x)
\]
which is $T_Y$-invariant. i.e. $DT_Y(W^{u/s}(x))= W^{u/s}(T_Yx)$.
Fix now the Calabi--Yau metric $\omega_Y$ on $Y$ (see \autoref{sssec:orbifold_ricci_flat_metrics}) and define the $\nu_*\mu$-measurable functions:
\begin{align}
		\label{eqn:rho_us_definition}
		\rho^u(x):= \log \norm{DT_Y\vert_{W^u(x)}}_{\omega_Y}, \quad \rho^s(x):= \log \norm{DT_Y\vert_{W^s(x)}}_{\omega_Y}
	\end{align}
where the norm of the operator is for the metrics $\omega_Y(x)$ and $\omega_Y(T_Yx)$ on the source and target tangent spaces.
The functions are $\nu_*\mu$-measurable since the spaces $W^{u/s}$ depend measurably on $x$.

\subsubsection{Computation in a basis}
	\label{sssec:computation_in_a_basis}
Let now $e_1\in W^u(x), e_2\in W^s(x)$ be a unimodular basis, i.e. using the dual basis $\sqrt{-1}e_1^\vee \wedge \ov{e_1}^\vee\wedge \sqrt{-1}e_2^\vee\wedge \ov{e_2}^\vee$ is the fixed volume form $\dVol_Y$.
There are two metrics on $\mathbf{T}_x Y$, one is $\omega_Y(x)$ and the other one is $\omega_{Y,Nh}(x)$, which is the pull-back by $T_Y^N$ of $\omega_Y(T_Y^Nx)$.

Suppose that in the fixed basis $\{e_1,e_2\}$ the metrics are represented by
\begin{align}
	\label{eqn:omega_0Nh_explicit}
	\omega_Y(x) =
	\begin{bmatrix}
		a_0 & b_0 \\
		\ov{b_0} & d_0 \\
	\end{bmatrix},
	\quad
	\omega_{Y,Nh}(x) =
	\begin{bmatrix}
		a_{Nh} & b_{Nh} \\
		\ov{b_{Nh}} & d_{Nh} \\
	\end{bmatrix}
\end{align}
and these matrices have determinant equal to $1$.
The relation to the functions defined in \autoref{eqn:rho_us_definition} is:
\[
	\rho^u(x) = \frac{1}{2} \log \left( \frac{a_h}{a_0} \right),
	\quad
	\rho^s(x) = \frac{1}{2} \log \left( \frac{d_h}{d_0} \right)
\]

\subsubsection{Defining $\beta$}
	\label{sssec:defining_beta}
Consider now the quantity
\begin{align}
	\label{eqn:definition_beta}
	\beta(x):= \frac{1}{2} \log(a_0d_0)
\end{align}
where $a_0,d_0$ are the entries in \autoref{eqn:omega_0Nh_explicit}.
Note that $\beta(x)\geq 0$ and it is independent of the earlier choice of unimodular basis.
Indeed $\log(a_0 d_0)$ can be expressed geometrically as follows.
The decomposition $\mathbf{T}_xY = W^s(x)\oplus W^u(x)$ and the metric $\omega_Y$ determine another metric $\widetilde{\omega_Y}$ defined as the restriction of $\omega_Y$ to $W^s$ and $W^u$ and declaring $W^s$ and $W^u$ to be orthogonal.
The log of the ratio of volume forms determined by $\omega_Y$ and $\widetilde{\omega_Y}$ is exactly $2\beta(x)=\log (a_0 d_0)$.
In particular $\beta$ is also a measurable function, since it is defined using standard constructions on measurable objects.

Finally, we have:
\begin{align}
	\label{eqn:the_sum_is_a_coboundary}
	T^*_Y\beta - \beta = \rho^s + \rho^u
\end{align}
which says that $\rho^s+\rho^u$ is a coboundary (in the dynamical sense).
Indeed, in explicit coordinates as above:
\[
	\rho^s(x) + \rho^u(x) = \frac{1}{2}\log (a_h d_h) - \frac{1}{2}\log (a_0 d_0)
\]
and the formula follows.

\subsection{The cohomological equation}
In this subsection we show that the expansion/contraction factors $\rho^u,\rho^s$ are cohomologous to constants via certain dynamical coboundary measurable functions $\alpha_s,\alpha_u$.
Crucially, $\alpha_s,\alpha_u$ will be shown to be bounded in \autoref{bounded}.

\subsubsection{Conditional measures}
  \label{sssec:conditional_measures}
To compare the restriction of the currents or \Kahler metrics to stable/unstable manifolds, it is useful to go through the intermediate notion of conditional measure.
Specifically, there exists by \cite[Thm. 9.4.1]{BarreiraPesin_NonuniformHyperbolicity} a partition $\zeta^s$ of $X$, with atom containing $x$ denoted $\zeta^s(x)$, with the following properties.
The partition is measurable, i.e. the atoms can be identified with the fibers of a measurable map $X\to X'$ to a standard measure space, and for $\nu_*\mu$-almost every $x\in X$ the atom $\zeta^s(x)$ is an open set in $\cW^s(x)$; such partitions are called \emph{subordinate} to the stable foliation.
Additionally the partition is contracted by the dynamics, i.e. $T_Y(\zeta^s(x))\subset \zeta^s(T_Y(x))$.
An analogous unstable partition $\zeta^u$ exists, but which is expanded by the dynamics.

Any measure on $X$ induces conditional (probability) measures on the atoms of $\zeta^{s/u}$.
By \cite[Cor.~G]{LY2} because $\mu$ is in the Lebesgue class, the conditional measures are also in the Lebesgue class.

\begin{proposition}[Cohomologous to a constant]
	\label{prop:cohomologous_to_a_constant}
In the setting of \autoref{main}, there exist $\nu_*\mu$-measurable functions $\alpha_s,\alpha_u$ such that
\begin{align}
\label{1}
	\nu_*\eta_+\vert_{\cW^u(x)} &= e^{-2\alpha_u}\omega_Y\vert_{\cW^u(x)}\\
\label{2}
	\nu_*\eta_-\vert_{\cW^s(x)} &= e^{-2\alpha_s}\omega_Y\vert_{\cW^s(x)}
\end{align}
for $\nu_*\mu$-a.e. $x\in Y$, where $\cW^{s/u}$ denote the stable/unstable manifolds.
Moreover the functions satisfy:
\begin{equation}
\label{eqn:rho_cohomologous_to_constant}
		\rho^u = T_Y^*\alpha_u - \alpha_u + \frac{h}{2},\quad \rho^s = T_Y^*\alpha_s - \alpha_s - \frac{h}{2}
\end{equation}
\end{proposition}
\begin{proof}
To check the existence of functions satisfying \autoref{1}, \autoref{2}, it suffices to show that for $\nu_*\mu$-a.e. $x$ the measure $\nu_*\eta_+\vert_{\cW^u(x)}$ is in the Lebesgue class, since then it must have a Radon-Nikodym derivative $e^{-2\alpha_u}$ relative to the measure $\omega_Y\vert_{\cW^u(x)}$.
But by \cite[Lemma 5.3]{deThelinDinh_DynamicsKahler}
the measures $\nu_*\eta_+\vert_{\cW^u(x)}$ are in the same measure class as the conditional measures of $\mu$ along the unstable foliation (regardless of any assumptions on $\mu$).
This property on $X$ pushes down to $Y$.
By the discussion in \autoref{sssec:conditional_measures}, the conditional measures are in the Lebesgue class on the atoms of their partitions $\nu_*\zeta^u$.
The existence of $\alpha_u,\alpha_s$ follows in the interior of the atoms of the partition.
Note finally that is $\zeta^u$ is a partition subordinated to the unstable foliation, then so is $T^i(\zeta^u)$ for any $i\geq 0$.
Since $\mu$-a.e. point has positive Lyapunov exponent, $\mu$-a.e. point will be in the interior of some atom of $T^i(\zeta^u)$ along the unstable foliation for some $i\geq 0$.
The relations in \autoref{eqn:rho_cohomologous_to_constant} follow from a computation:
\begin{multline}
e^{2\rho^u-2T_Y^*\alpha_u}\omega_Y\vert_{\cW^u(x)}=
e^{-2T_Y^*\alpha_u}\omega_{Y,h}\vert_{\cW^u(x)}=\\
= \nu_*T^*\eta_+\vert_{\cW^u(x)}=
e^h\nu_*\eta_+\vert_{\cW^u(x)} = e^{h-2\alpha_u}\omega_Y\vert_{\cW^u(x)}
\end{multline}
and similarly for $\rho^s$.
\end{proof}

Note also that combining \autoref{eqn:the_sum_is_a_coboundary} and \autoref{prop:cohomologous_to_a_constant} we see that
$$T_Y^*(\alpha_u+\alpha_s-\beta)-(\alpha_u+\alpha_s-\beta)=\rho^u+\rho^s-(\rho^u+\rho^s)=0,$$
and so by ergodicity of $\mu$:
\begin{equation}\label{sumbeta}
\alpha_u+\alpha_s=\beta+\delta,
\end{equation}
$\mu$-a.e., for some constant $\delta\in\mathbb{R}$.

The next proposition is analogous to \cite[Thm. 5.1]{CD}, where it is proved by a different argument involving renormalization along the stable manifolds.
\begin{proposition}[The restricted current is flat]
	\label{prop:the_restricted_current_is_flat}
	For $\nu_*\mu$-a.e. $x$ we have that
	\[
		(\xi_x^u)^*\nu_*\eta_+ = e^{-2\alpha_u(x)}\omega_Y(x)\big\vert_{W^u(x)}
	\]
	i.e. the pulled back current to the unstable tangent space is a flat metric.
\end{proposition}
\begin{proof}
	Consider the function $f(x,r)$ defined for $\nu_*\mu$-a.e. $x$ and $r>0$:
	\[
		f(x,r):= \frac{\int_{W^u(x,r)}(\xi_x^u)^*\nu_*\eta_+}{\pi r^2}
	\]
	where $W^u(x,r)$ denotes the ball of radius $r$ for the metric $e^{-2\alpha_u(x)}\omega_Y(x)$ on $W^u(x)$.
  Recall also that we normalized the derivative of $\xi_x^{u}$ at $0$ to be the identity.

	Therefore, by the Lebesgue density theorem combined with \autoref{1}, we have that $\lim_{r\to 0}f(x,r)=1$ for $\nu_*\mu$-a.e. $x$.
	Applying the automorphism and using the coboundary property, we have that
	\[
		f(T_Yx,e^{h}r) = f(x,r) \text{ or equivalently } f(T_Y^{-1}x,e^{-h}r)= f(x,r).
	\]
  Now for any $\ve>0$ there exists $r_\ve>0$ such that the set
  $$Y_\ve:=\left\lbrace x\colon |1-f(x,r)| \leq \ve, \forall r\in (0,r_\ve)\right\rbrace$$
  has positive Lebesgue measure.
  By ergodicity of $T_Y$, $\nu_*\mu$-a.e. $x$ visits $Y_\ve$ for arbitrarily large negative times, so combined with the above equation it follows that $f(x,r) = 1$ for $\nu_*\mu$-a.e. $x$ and any $r>0$.

	Take now a point $y\in \cW^u(x)$ with $v_y\in W^u(x)$ such that $y=\xi_x^u(v_y)$.
	We know that the composition $(\xi_y^{u})^{-1}\circ \xi_x^u$ is an affine map and let $C_{x,y}$ be the derivative of this map for the metrics $e^{-2\alpha_u(p)}\omega_Y(p)$ with $p=x,y$.
	It suffices to show that $C_{x,y}=1$, since then it follows that $\cW^u(p)$ carries a canonical flat metric independent of $p$, and the Radon-Nikodym derivative of $(\nu_*\eta_+)\vert_{\cW^u(p)}$ relative to this metric is identically $1$ (because $f(x,r)=1$).
  By symmetry of the next argument in the two variables, $C_{x,y}\leq 1$ also suffices.

	Since $f(y,r)=1$ for $y$ in a set of full Lebesgue measure on $\cW^u(x)$ (and we pick $y$ in this set) it follows that
	\[
		\frac{\int_{W^u(y,r)}(\xi_y^u)^*\nu_*\eta_+}{\pi r^2} = 1 \text{ for all }r>0.
	\]
	Transporting this identity back to $W^u(x)$ using $(\xi_x^{u})^{-1}\circ\xi^u_{y}$ and using that this map takes a ball of radius $r$ in $W^u(y)$ to a ball of radius $\frac{r}{C_{xy}}$ in $W^u(x)$, it follows that:
	\[
		\frac{\int_{W^u(v_y,r)}(\xi_x^u)^*\nu_*\eta_+}{\pi r^2} = C^2_{xy}
	\]
	where $W^u(v_y,r)$ denotes the ball of radius $r$ at $v_y$ in the metric $e^{-2\alpha_u(x)}\omega_Y(x)\vert_{W^u(x)}$.
	But $W^u(v_y,r)\subset W^u(x,r+|v_y|))$ and $\eta_+\geq 0$ so we have:
	\[
		1 = \frac{\int_{W^u(x,r+|v_y|)}(\xi_x^u)^*\nu_*\eta_+}{\pi (r+|v_y|)^2} \geq \frac{\int_{W^u(v_y,r)}(\xi_x^u)^*\nu_*\eta_+}{\pi (r+|v_y|^2)} = C^2_{xy} \frac{r^2}{(r+|v_y|)^2}
	\]
	Letting $r\to \infty$ the desired conclusion follows.
\end{proof}

\subsection{Boundedness of the dynamical coboundaries}
The next step, which was hinted to earlier, consists in showing that the dynamical coboundary functions $\alpha_u$ and $\alpha_s$ are in fact bounded:
\begin{theorem}\label{bounded}
The coboundary functions $\alpha_u,\alpha_s$ belong to $L^\infty(\nu_*\mu)$.
\end{theorem}

\begin{remark}
	\label{rmk:vanishing_pullback}
	Before the proof of \autoref{bounded}, note that for $\mu$-a.e. $x$ we have that $(\xi^u_x)^*\nu_*\eta_-=0$ by say \cite[Thm. 5.3]{Cantat_K3}, or a combination of \cite[Cor.~4.2, Prop.~4.4]{deThelinDinh_DynamicsKahler} in the non-projective case.
	We will restrict to such $x$ and their image in $Y$ and will eventually construct an entire curve on $Y$ for which the pullback of both $\nu_*\eta_+$ and $\nu_*\eta_-$ vanish.
  This idea is also a key point in the work of Cantat--Dupont \cite{CD}.
\end{remark}

\begin{remark}\label{potentials}
Recall that the map $\nu\colon X\to Y$ contracts the analytic subset $V\subset X$ to the singular points of $Y$, and that by \autoref{prop:kahler} we can write $[\eta_+]+[\eta_-]=\nu^*[\omega_Y]$ for a Ricci-flat orbifold K\"ahler metric $\omega_Y$ on $Y$.
For later use, the following observation will be useful: on $Y$ we can write
\begin{equation}\label{potentials2}
\nu_*(\eta_++\eta_-)=\omega_Y+\ddbar\vp,
\end{equation}
where $\vp\in C^0(Y)$.

To see this, we fix smooth representatives $\alpha_{\pm}$ of $[\eta_{\pm}]$ on $X$, and then we can then write $\eta_{\pm}=\alpha_{\pm}+\ddbar\vp_{\pm}$ where $\vp_{\pm}$ are H\"older continuous functions on $X$, as recalled in \autoref{sssec:the_eigencurents}.
On the other hand we can also write $\nu^*\omega_Y=\alpha_++\alpha_-+\ddbar u$ for some continuous function $u$ on $X$.
It follows that
$$\eta_++\eta_-=\nu^*\omega_Y+\ddbar(\vp_++\vp_--u),$$
and restricting this to any irreducible component $C$ of $V$ (which as we know is contracted to a point by $\nu$) we have
$$0\leq (\eta_++\eta_-)|_C=\ddbar(\vp_++\vp_--u)|_C,$$
so $\vp_++\vp_--u$ is a plurisubharmonic function on the compact curve $C$, hence constant. It follows that we have $\vp_++\vp_--u=\nu^*\vp$ for some continuous function $\vp$ on $Y$ such that $\omega_Y+\ddbar\vp$
is a closed positive current on $Y$ which satisfies $\nu^*(\omega_Y+\ddbar\vp)=\eta_++\eta_-$, and \eqref{potentials2} follows from this.
\end{remark}

\begin{proof}[Proof of \autoref{bounded}]
By \autoref{sumbeta} together with the fact that $\beta\geq 0$, it suffices to show that $\alpha_u\leq C, \alpha_s\leq C$ on a $\nu_*\mu$-full measure set.
We give the argument for $\alpha_u$, the one for $\alpha_s$ being identical.

Suppose that $\alpha_u$ is not bounded on the full $\nu_*\mu$-measure set of points which are not orbifold singularities, for which unstable manifolds exist, and on which \autoref{prop:the_restricted_current_is_flat}, \autoref{prop:cohomologous_to_a_constant}, and \autoref{rmk:vanishing_pullback} hold.
So there is a sequence of such points $x_i$ with $\alpha_u(x_i)\to+\infty$.

Therefore the unstable parametrizations $\xi^u_{x_i}\colon W^u(x_i)\to Y$ satisfy $\xi^u_{x_i}(0)=x_i, D_0\xi_x^u=\id$ and
\[
	(\xi_{x_i}^u)^*\nu_*\eta_+ = e^{-2\alpha_u(x_i)}\omega_Y(x_i)\vert_{W^u(x_i)},	
\]
which goes to zero.
Fix a sequence $R_i\to\infty$.

Suppose first that there is $C$ such that
$$\sup_i\sup_{D_{R_i}(0)}|D\xi^u_{x_i}|\leq C,$$
relative to $\omega_Y$ on $Y$ and the flat metric $\omega_Y(x_i)$ on $W^u(x_i)$.
Then
by Ascoli--Arzel\`a up to passing to a subsequence, the maps $\xi_{x_i}^u|_{D_{R_i}(0)}$ converge locally uniformly to a nonconstant entire curve $\xi:\mathbb{C}\to Y$.
By construction $\alpha_u(x_i)\to +\infty$, the current $\nu_*(\eta_++\eta_-)$ has locally continuous potentials (see \autoref{potentials}), and the convergence $\xi_i\to \xi$ is locally uniform, we can exchange limits to conclude that $\xi^*(\nu_*(\eta_++\eta_-))=0$ {(using \autoref{rmk:vanishing_pullback}).}

If there is no such $C$, we can find points $z_i\in \ov{D_{R_i}(0)}$ such that $|D\xi^u_{x_i}|(z_i)\to\infty$ (up to subsequence).
Working in $D_1(z_i),$ we apply the standard Brody reparametrization argument.
Namely for each $i$ we pick a point $y_i\in D_1(z_i)$ which maximizes
$$\delta_i(z)=\dist(z,\de D_1(z_i))|D\xi^u_{x_i}|(z).$$
Call $r_i=\dist(y_i,\de D_1(z_i)), a_i=|D\xi^u_{x_i}|(y_i),$ so that
$$a_i\geq a_ir_i\geq \delta_i(z_i)=|D\xi^u_{x_i}|(z_i)\to\infty,$$
while for all $z\in D_{r_i/2}(y_i)$ we have
$$a_ir_i\geq \dist(z,\de D_1(z_i))|D\xi^u_{x_i}|(z)\geq \frac{r_i}{2}|D\xi^u_{x_i}|(z),$$
hence $|D\xi^u_{x_i}|\leq 2a_i$ on $D_{r_i/2}(y_i)$.

Let now
\[
	\ti{\xi}_i\colon D_{a_ir_i/2}(0)\to Y,\quad \ti{\xi}_i(z)=\xi^u_{x_i}\left(y_i+\frac{z}{a_i}\right),	
\]
which are defined on bigger and bigger discs and satisfy
$$\sup_{D_{a_ir_i/2}(0)}|D\ti{\xi}_i|\leq 2,\quad |D\ti{\xi}_i|(0)=1,$$
$$\ti{\xi}_i^* \nu_*(\eta_+) = \frac{e^{-2\alpha_u(x_i)}}{a_i^2}\omega_Y(x_i)\vert_{W^u(x_i)},$$
which goes to zero.
Again by Ascoli--Arzel\`a up to passing to a subsequence, the maps $\ti{\xi}_i$ converge locally uniformly to a nonconstant entire curve $\xi\colon \mathbb{C}\to Y$.
Since $a_i\to+\infty$, reasoning as in the case above we again conclude that $\xi^*(\nu_*(\eta_++\eta_-))=0$.

In both cases, the existence of such an entire curve $\xi$ is a contradiction to \autoref{entire} below.
\end{proof}

In the proof above we used the following proposition, which is due to Dinh--Sibony \cite{DS} in general (see \cite[Proposition 3.10]{CD}).
In the $K3$ case we can give a very simple proof:
\begin{proposition}\label{entire}
There is no nonconstant entire curve $\xi:\mathbb{C}\to Y$ such that
$$\xi^*(\nu_*(\eta_++\eta_-))=0.$$
\end{proposition}
\begin{proof}
Recall from \autoref{potentials} that on $Y$ we can write
$$\nu_*(\eta_++\eta_-)=\omega_Y+\ddbar\vp,$$
where $\omega_Y$ is a Ricci-flat orbifold K\"ahler metric on $Y$ and $\vp\in C^0(Y)$.

For every $r>0$ choose $\chi_r$ a nonnegative radial cutoff function which equals $1$ on $D_r$, is supported in $D_{2r}$, and such that $|\ddbar\chi_r|\leq \frac{C}{r^2},$ for a constant $C$ independent of $r$. This can be done by letting
$$\chi_r(z)=\eta\left(\frac{|z|}{r}\right),$$
where $\eta$ is a nonnegative cutoff function on $\mathbb{R}_{\geq 0}$ which equals $1$ on $[0,1]$ and vanishes on $[2,\infty)$. We then compute
$$\ddbar\chi_r=i\de\left(\eta'\cdot \frac{z}{2|z|r}d\ov{z}\right)=\left(\frac{\eta''}{4r^2}+\frac{\eta'}{4r|z|}\right)idz\wedge d\ov{z},$$
which is nonzero only for $r\leq |z|\leq 2r$ and so satisfies $|\ddbar\chi_r|\leq \frac{C}{r^2}$ everywhere.
Then for all $r>0$ we have
$$0=\int_{D_{2r}}\chi_r\xi^*(\nu_*(\eta_++\eta_-))=\int_{D_{2r}}\chi_r\xi^*(\omega_Y+\ddbar\vp),$$
and so the area of $\xi(D_r)$ is bounded above by
$$\int_{D_{2r}}\chi_r\xi^*\omega_Y=\int_{D_{2r}}\chi_r \xi^*(-\ddbar\vp)=\int_{D_{2r}}(-\vp\circ\xi)\ddbar\chi_r\leq \frac{Cr^2}{r^2}\leq C,$$
and so the entire curve $\xi$ has finite area, hence it extends to a holomorphic map $\xi\colon \mathbb{P}^1\to Y$, see Moncet \cite[Prop. 3.31]{Moncet}.
If $\xi$ is nonconstant, then by construction the rational curve $C=\xi(\mathbb{P}^1)$ satisfies $C\cdot([\nu_*(\eta_++\eta_-)])=0$, which is impossible since $[\nu_*(\eta_++\eta_-)]=[\omega_Y]$ is K\"ahler.
\end{proof}

\subsection{Completion of the proof of the main results}
In this subsection we complete the proof of \autoref{main3} (and therefore also of \autoref{main}), and we also derive \autoref{coro}.

\begin{proof}[Proof of \autoref{main3}]
From \autoref{bounded} we know that the coboundaries $\alpha_u,\alpha_s$ belong to $L^\infty(\nu_*\mu)$.
Since $\rho^s(x)=\log\|DT_Y|_{W^s(x)}\|_{\omega_Y}$, it follows that for all $N\geq 1$ we have
\[\begin{split}
\log \norm{D_xT^{N}_Y\vert_{W^s(x)}}_{\omega_Y}&=\rho^s(x)+\cdots+\rho^s(T_Y^{N-1}x)\\
&={\alpha_s(T_Y^{N}x)-\alpha_s(x)}-\frac{Nh}{2}\\
&\leq -\frac{Nh}{2}+C,
\end{split}\]
using the coboundary relation \eqref{eqn:rho_cohomologous_to_constant} and the $L^\infty$ bound for $\alpha_s$.
This proves \eqref{exp2}, and a similar argument shows \eqref{exp1}.

The $L^\infty$ bound for $\alpha_u,\alpha_s$ together with $\beta=\alpha_u+\alpha_s-\delta$ (by \autoref{sumbeta}) show that $\beta$ is uniformly bounded.
From \autoref{eqn:omega_0Nh_explicit} and the definition of $\beta$ in \autoref{eqn:definition_beta}, a uniform upper bound on $\beta$ gives a uniform lower bound on the angle between $W^u$ and $W^s$, measured relative to $\omega_Y$.
\end{proof}

Lastly, we prove \autoref{coro}:
\begin{proof}[Proof of \autoref{coro}]
Cantat has proved \cite{Ca2} that if a compact complex surface $X$ admits an automorphism $T$ with positive topological entropy then $X$ is either a torus, K3, Enriques, a blowup of these, or rational. These are all projective, except for non-projective tori and $K3$ and their blowups, so thanks to Cantat-Dupont \cite{CD} (which assume projectivity) we may assume that $X$ is a torus, a $K3$ surface or a blowup of these.

First, if $X$ is a torus then $T$ is induced by an affine transformation of $\mathbb{C}^2$, and therefore $(X,T)$ is trivially a Kummer example.
Second, if $X$ is $K3$ then the result follows from our main \autoref{main}. Lastly, if $X$ is a blowup of a torus or $K3$ surface, say $\pi\colon X\to Y$ is the sequence of blowups, then $T$ induces a bimeromorphic map $T_Y$ of $Y$, which must be a biholomorphism (see e.g. \cite[Prop.~III.4.6]{BHPV}).
Furthermore the topological entropy of $T_Y$ equals the one of $T$ by the same argument as in \cite[Lemma 2.2]{Ka}.
From the relation $T_Y\circ\pi=\pi\circ T$ we deduce that the eigencurrents $\eta_\pm$ for $T$ on $X$ are equal to the pullbacks of the corresponding eigencurrents for $T_Y$ on $Y$, hence the measures of maximal entropy satisfy $\mu=\pi^*\mu_Y$. Since by assumption $\mu$ is absolutely continuous with respect to the Lebesgue measure, we conclude that $\mu_Y$ is also absolutely continuous with respect to the Lebesgue measure. By the previous cases, we see that $(Y,T_Y)$ is a Kummer example, and hence by definition so is $(X,T)$.
\end{proof}


\section{Alternative arguments for the existence of the coboundaries}
	\label{sec:alternative_arguments_for_the_existence_of_the_coboundaries}

In this section we give an alternative argument for the existence of the expansion/contraction coboundaries in \autoref{prop:cohomologous_to_a_constant}.
Instead of using the theories of Pesin and Ledrappier--Young {and the result by De Th\'elin--Dinh,} we rather exploit our specific geometric setup, and in particular the hyperbolic geometry of the space of hermitian metrics on a tangent space with given volume form.
The coboundaries thus constructed have better integrability properties than those given in \autoref{prop:cohomologous_to_a_constant}, which are just measurable.
We hope that the ideas below may prove useful in related problems.

To simplify notation, in this section we work on $X$ as if it was the orbifold $Y$, so that $[\eta_+]+[\eta_-]$ is a K\"ahler class and we will write $\omega_{Nh}$ for the Ricci-flat metrics on $X$ that play the role of $\omega_{Y,Nh}$, so that $\omega_0$ replaces $\omega_Y$. In general one would apply the arguments below to the orbifold $Y$.

\subsection{A simple Lemma}
	\label{ssec:embarrassing_lemma}
The following simple lemma is reminescent of the Gottschalk--Hedlund theorem.
\begin{lemma}[Finding a coboundary]
	\label{lem:finding_a_coboundary}
Let $T\colon (X,\mu)\to (X,\mu)$ be a mixing transformation of a probability measure space.
	Suppose that for $f\in L^2(\mu)$ there exists $C\geq 0$ such that
	\begin{equation}\label{cond}
		\norm{f + T^* f + \cdots +(T^n)^*f}_{L^2(\mu)} \leq C,
	\end{equation}
	for all $n\geq 1.$ Then there exists $h\in L^2(\mu)$ such that $f = h - T^* h$.
\end{lemma}
\begin{proof}
	First note that \autoref{cond}, combined with the von Neumann ergodic theorem, implies that $\int_X fd\mu = 0$.
	By the mixing property of $T$, this implies that the only possible weak limit of $(T^n)^*f$ in $L^2(\mu)$ is $0$.

	Thanks again to the uniform $L^2$ boundedness of the Birkhoff sums of $f$ in \autoref{cond}, there is some weak limit $h\in L^2(\mu)$ of $f + T^* f + \cdots + (T^{n_j})^* f$ along some subsequence $\left\lbrace n_j \right\rbrace$.

	Then using the above remark that the weak limit of $(T^{n_j+1})^* f$ is $0$, it follows that
	\[
		(T^* h) = \lim_{n_j} \left[ \left( f + T^*f \cdots + (T^{n_j+1})^*f \right)  -f \right] = h- f
	\]
	which is the desired conclusion.
  Note that $T^*$ is weakly continuous, since it is an isometry and hence has an adjoint (which can be used to obtain weak continuity).
\end{proof}

Recall that for automorphisms of $K3$ surfaces with positive entropy, the measure of maximal entropy is mixing by \cite[Theorem 6.1]{Cantat_K3}.

\subsection{Yau vs Oseledets curve}
	\label{ssec:yau_vs_oseledets_curve}

\subsubsection{Hyperbolic geometry}
	\label{sssec:hyperbolic_geometry}

Let $V$ be a complex $2$-dimensional vector space, equipped with a non-degenerate complex volume form $\Omega\in \Lambda^2(V^\vee)$.
Let $\bH^3(V)$ be the space of hermitian metrics which induce the same (real) volume on $V$ as $\Omega$; this space is naturally isomorphic to real hyperbolic $3$-space, since it can be described as the quotient $\SL_2(\bC)/\SU(2)$.

Consider now a decomposition $V= W_+ \oplus W_-$ into two complex lines.
This determines a subset $\gamma\subset \bH^3(V)$ of hermitian metrics for which the decomposition of $V$ is orthogonal.
There is a natural nearest point projection $\pi_\gamma:\bH^3(V)\to \gamma$, which we now make explicit.
Note that $\gamma$ is a geodesic for the hyperbolic metric.

\subsubsection{Working in coordinates}
	\label{sssec:working_in_coordinates}
Assume that we have $\bC^2 = \bC \oplus \bC$ as our decomposition.
A hermitian metric, inducing the standard volume form, is given by a $2\times 2$ matrix:
\[
	\begin{bmatrix}
		a & b \\
		\ov{b} & d \\
	\end{bmatrix}
			\text{ with }a,d \in \bR, b\in \bC \text{ and } ad - |b|^2=1.
\]
The metrics on $\gamma$, for which the decomposition is orthogonal, have $b=0$.

By symmetry considerations, we must have
\[
 	\pi_\gamma \left( \begin{bmatrix}
 		a & b \\
 		\ov{b} & a \\
 		\end{bmatrix} \right) =
	\begin{bmatrix}
		1 & 0 \\
		0 & 1 \\
	\end{bmatrix}
\]
since the map $\pi_\gamma$ is equivariant for the action of the matrices in {$\GL_2(\bC)$} which preserve the decomposition $\bC^2 = \bC \oplus \bC$, and the transformation exchanging the two axes is in there.

Finally, using equivariance under the action of scaling the coordinates by $e^{t/2}$ and $e^{-t/2}$ respectively, it follows that:
\[
	\pi_\gamma \left(
	\begin{bmatrix}
		e^{t/2} & 0 \\
		0 & e^{-t/2} \\
	\end{bmatrix}
	\begin{bmatrix}
		a & b \\
		\ov{b} & a \\
	\end{bmatrix}
		\begin{bmatrix}
		e^{t/2} & 0 \\
		0 & e^{-t/2} \\
	\end{bmatrix}
 	\right)
 	= \pi_\gamma \left(
 	\begin{bmatrix}
 		e^t a & b \\
 		\ov{b} & e^{-t} a \\
 	\end{bmatrix}
 	\right)
 	=
 	\begin{bmatrix}
 		e^t & 0 \\
 		0 & e^{-t} \\
 	\end{bmatrix}
\]
which can be rewritten as
\[
	\pi_\gamma \left( \begin{bmatrix}
		a & b \\
		\ov{b} & d \\
	\end{bmatrix}
	 \right)
	 =
	 \begin{bmatrix}
	 	\sqrt{a/d} & 0 \\
	 	0 & \sqrt{d/a} \\
	 \end{bmatrix} 		
\]

\subsubsection{Hyperbolic Distances}
	\label{sssec:hyperbolic_distances}
Given two Hermitian forms $\omega_1,\omega_2$ on $V$ (compatible with the volume) the distance in hyperbolic space between them is defined by
\[
	\frac{\omega_1\wedge\omega_2}{\Omega\wedge \ov{\Omega}} = \frac 12 \left( e^{\dist(\omega_1,\omega_2)} + e^{-\dist(\omega_1,\omega_2)} \right).
\]
Equivalently, one can pick a basis in which $\omega_1$ is standard Euclidean and diagonalize $\omega_2$ using the spectral theorem to define the distance as the logarithm of the largest (relative) singular value.

\subsubsection{Yau and Oseledets curves}
We now apply the discussion above to $V=\mathbf{T}_xX$ with $x$ in the $\mu$-full measure set where the Oseledets theorem gives us the decomposition $\mathbf{T}_xX=W^s(x)\oplus W^u(x)$.
This determines the Oseledets curve $\gamma\subset\bH^3(V)$, a hyperbolic geodesic.
On the other hand, by Yau's Theorem \cite{Ya} we have Ricci-flat metrics $\omega_t$ on $X$ in the class $e^t[\eta_+]+e^{-t}[\eta_-]$, which {together define the Yau curve} in $\bH^3(V)$ (although we will be mostly interested only in the values $t=Nh$, $N\geq 1$).

Recall that the expansion/contraction factors defined in \autoref{eqn:rho_us_definition} are equal to
\[
	\rho^u(x) = \frac{1}{2} \log \left( \frac{a_h}{a_0} \right),
	\quad
	\rho^s(x) = \frac{1}{2} \log \left( \frac{d_h}{d_0} \right)
\]
and so a telescoping sum gives
\[
	\rho^u(x) + \rho^u(Tx) + \cdots + \rho^u(T^{N-1}x) = \frac{1}{2} \log \left( \frac{a_{Nh}}{a_0} \right)
\]
and similarly for $\rho^s(x)$.

We will later need the following:
\begin{proposition}
	\label{prop:integrals_of_rho_us}
	We have the identities:
	\begin{align}
	\label{eqn:rho_integral}
	\begin{split}
		\int_X \rho^u(x)\, d\mu(x) & = \frac{h}{2}\\
		\int_X \rho^s(x)\, d\mu(x) & = -\frac{h}{2}
	\end{split}
	\end{align}	
\end{proposition}
\begin{proof}
  Recall that the Lyapunov exponent of $\mu$ is $h/2$, see \autoref{sssec:absolute_continuity}.

	Once the stable and unstable bundles are given, and since in our case they are line bundles, the exponent can be computed from the formula:
	\[
		\lambda = \int_X \log \norm{DT|_{W^u(x)}} d\mu(x)
	\]
	where we compute the norm of $DT\colon W^u(x)\to W^u(Tx)$ for one fixed ambient metric.
	Note that if the metric is changed, then the quantity $\log \norm{DT|_{W^u(x)}}$ changes by a coboundary, i.e. $\alpha(x)-\alpha(Tx)$ where $e^{\alpha(x)}$ is the constant of proportionality between the old and the new metric, when restricted to $W^u(x)$.
	In particular, the integral is independent of the metric.

	Since by definition $\rho^u(x)$ is the pointwise norm of $DT$ on the unstable for the Ricci-flat metric, the claimed identity follows.
\end{proof}

By the discussion in \ref{sssec:hyperbolic_geometry} there are also the ``Oseledets-projected'' metrics $\theta_0, \theta_{Nh}$ which correspond to the projection of $\omega_0,\omega_{Nh}$ to the geodesic $\gamma$ determined by the stable/unstable decomposition.
In the fixed basis $\{e_1,e_2\}$ as in \ref{sssec:computation_in_a_basis}, the metrics are:
\begin{align}
	\label{eqn:theta_0Nh_Explicit}
	\theta_0 =
	\begin{bmatrix}
		\sqrt{\frac{a_0}{d_0}} & 0 \\
		0 & \sqrt{\frac{d_0}{a_0}} \\
	\end{bmatrix},
	\quad
	\theta_{Nh} =
	\begin{bmatrix}
		\sqrt{\frac{a_{Nh}}{d_{Nh}}} & 0 \\
		0 & \sqrt{\frac{d_{Nh}}{a_{Nh}}} \\
	\end{bmatrix}
\end{align}
The distance-decreasing property of projections in the hyperbolic metric on $\bH^3(V)$, {together with the formula for the hyperbolic distance in \autoref{sssec:hyperbolic_distances}}, gives
\begin{align}
	\label{eqn:lambda_controls_theta}
	\dist(\theta_0,\theta_{Nh}) \leq \dist(\omega_0,\omega_{Nh}) = 2 \lambda(x,N).
\end{align}

The distance $\dist(\theta_0, \theta_{Nh})$ is computed explicitly as:
\begin{equation}\label{dista}
\begin{split}
\dist(\theta_0,\theta_{Nh}) &= \frac{1}{2} \left| \log \left( \frac{a_{Nh}d_0}{a_0 d_{Nh}} \right) \right| \\
&=\left| \rho^u(x) + \cdots + \rho^u(T^{N-1}x)  - \rho^s(x) - \cdots - \rho^s(T^{N-1}x)\right|\\
& = |S_N \rho^u(x) - S_N \rho^s(x)|
\end{split}
\end{equation}
at $\mu$-a.e. point, where $S_N f $ denotes the Birkhoff sum of the function $f$,
\[
	S_Nf(x) := f(x) + \cdots + f(T^{N-1}x).
\]

Combining this identity with the previous inequality gives
\begin{align}
	\label{eqn:Birkhoff_sum_bound}
	|S_N \rho^u(x) - S_N \rho^s(x)| \leq 2 \lambda(x,N)
\end{align}
The following observation is going to be crucial:

\begin{proposition}\label{integr}Suppose that $\mu$ is absolutely continuous with respect to the Lebesgue measure, so that $\mu=\frac{\dVol}{|S|}\bigg|_S.$
Then the Birkhoff sums of $\rho^u-\rho^s-h$ satisfy
\begin{equation}\label{exp+}
\int_X e^{S_N (\rho^u-\rho^s - h)}d\mu\leq \frac{2}{|S|},\text{ for all }N\geq 1.
\end{equation}

\end{proposition}
\begin{proof}
We have
\[\begin{split}
2e^{Nh}&\geq e^{Nh} + e^{-Nh} = \int_X \omega_0 \wedge \omega_{Nh}\\
&=\int_X (e^{2\lambda(x,N)}+e^{-2\lambda(x,N)})\dVol\geq \int_X e^{2\lambda(x,N)}\dVol,
\end{split}\]
i.e.
\begin{equation}\label{expl}
\int_X e^{2\lambda(x,N)-Nh}\dVol\leq 2,
\end{equation}
but from \autoref{eqn:Birkhoff_sum_bound} we also have
$S_N \rho^u(x) - S_N \rho^s(x) \leq 2 \lambda(x,N)$,
and so
\[\int_X e^{S_N (\rho^u-\rho^s - h)}\dVol=\int_X e^{S_N \rho^u - S_N \rho^s-Nh}\dVol\leq 2.\]
This finally gives:
\[\int_X e^{S_N (\rho^u-\rho^s - h)}d\mu\leq \frac{2}{|S|}.\]
\end{proof}

Next, observe that both $\rho^u$ and $\rho^s$ are in $L^\infty$, with a uniform bound which only depends on the transformation $T$ and the Ricci-flat metric $\omega_0$. In particular the function $\rho^u-\rho^s-h$ is also in $L^\infty$, and thanks to \autoref{prop:integrals_of_rho_us} it satisfies
\begin{equation}\label{exp0}
\int_X(\rho^u-\rho^s-h)d\mu =0.
\end{equation}

\subsection{Exponential integrability of Birkhoff sums}
Our goal now is to use \autoref{exp+} and \autoref{exp0} to prove an exponential integrability bound for $|S_N(\rho^u-\rho^s-h)|$. As it turns out this is an essentially formal consequence of \autoref{exp+} and \autoref{exp0}, as we now show.

\begin{proposition}[Exponential integrability of Birkhoff sums]
	\label{prop:exponential_integrability_of_birkhoff_sums}
Let $T\colon (X,\mu)\to (X,\mu)$ be an invertible transformation of a probability measure space.
	Suppose that $f\in L^1(X,\mu)$ with $\int_X f\, d\mu=0$ is such that
	\begin{equation}\label{assn}
		\int_X e^{S_Nf}d\mu \leq C
	\end{equation}
	for a uniform constant $C$ and all $N\geq 1$. Then
  for every $0<\gamma<\frac{1}{6}$ there is $C'=C'(C,\gamma)$, {which does not depend on $f$,} such that
  \begin{equation}
    \label{concl}
  \int_X e^{\gamma|S_Nf|}d\mu\leq C', \text{ for all }N\geq 1.
  \end{equation}
\end{proposition}
\begin{proof}
	Decompose $f=f^+ - f^-$ into its positive and negative  parts, and for Birkhoff sums denote by $S_Nf = S_N^+f - S_N^-f$ the decomposition into positive and negative parts.
Suppose we show that there is a constant $C'$ that depends only on $C$ such that the negative part $f^-$ satisfies the bound
		\begin{equation}\label{topr}
			\mu\left( \left\lbrace x:f^-(x)\geq L \right\rbrace \right) \leq C'e^{-\frac{L}{6}}
		\end{equation}
		for all $L\geq 0$.

Then this can be applied to the function $S_Nf$ with the transformation $T^N$, which would thus give us
	\[
			\mu\left( \left\lbrace x: S_N^-f(x)\geq L \right\rbrace \right) \leq C'e^{-\frac{L}{6}}
		\]
		for all $L\geq 0$ and all $N\geq 1$.
    On the other hand \autoref{assn} directly implies
\begin{equation}\label{trivial}
\int_X e^{S^+_Nf}d\mu \leq C,
\end{equation}
which together with the Chebyshev-Markov inequality implies that
		\begin{equation}\label{cheb}
			\mu \left( \left\lbrace x \colon S_N^+f(x)\geq L \right\rbrace \right) \leq Ce^{-L}
		\end{equation}
for all $L\geq 0$ and all $N\geq 1$, and so \autoref{concl} follows from these bounds together with the elementary formula
$$\int_X e^{\gamma |u|}d\mu=\gamma\int_0^\infty \mu(\{|u|\geq t\})e^{\gamma t} dt.$$

So it suffices to prove \autoref{topr}. Define the set of interest as
	\[
		B_L := \left\lbrace x \colon f^-(x)\geq L \right\rbrace
	\]
	Consider now the set where the positive parts of the Birkhoff sums are large (but on a smaller scale):
	\[
		P_{j,L} := \left\lbrace x \colon S_j^+f(x)\geq \frac{1}{2}(L-1) \right\rbrace
	\]
	which by \autoref{cheb} satisfy $\mu(P_{j,L})\leq C'e^{-\frac{L}{2}}$.
	Set
	\[
		C_L := B_L \setminus \bigcup_{i=-e^{\frac{L}{6}}}^{e^{\frac{L}{6}}} \bigcup_{j=1}^{e^{\frac{L}{6}}} T^i (P_{j,L})
	\]
	which satisfies the size bound:
	\[
		\mu(C_L) \geq \mu(B_L) - 3e^{\frac{L}{3}}\cdot C\cdot e^{-\frac{L}{2}}=\mu(B_L)-C'e^{-\frac{L}{6}}.
	\]
	Any point $x\in C_L$ has the property that $f^-(x)\geq L$ and for any other point in its orbit $T^ix$ (with $i=-e^{\frac{L}{6}}\ldots e^{\frac{L}{6}}$) any Birkhoff sum (with $j=1\ldots e^{\frac{L}{6}}$) satisfies $S_j^+f(T^ix)\leq \frac{L-1}{2}$.

	Consider now any point $y$ in the support of the function
	\[
		g_L:=\id_{C_L} + \id_{T^{-1}C_L} + \cdots + \id_{T^{-e^{\frac{L}{6}}}C_L}
	\]
	Then $g_L(y)$ is the number of visits of $y$ to $C_L$ in the times $1\ldots e^{\frac{L}{6}}$.
	The Birkhoff sum $S_{e^{\frac{L}{6}}}f(y)$ can be divided into at most $g_L(y)+1$ intervals where the positive part is bounded above by $\frac{L-1}{2}$ by the construction of $C_L$, and the $g_L(y)$ points where $f^-\geq L$.
	This implies that
	\[
		S_{e^{\frac{L}{6}}}^-f(y) \geq g_L(y)\left(L-2\frac{L-1}{2}\right)=g_L(y)
	\]
	for all $y$ in the support of $g_L$.
	Integrating over all $y$ gives:
	\begin{equation}\label{numbered}
		\int_{\mathrm{Spt}(g_L)} S_{e^{\frac{L}{6}}}^-f(y) d\mu(y) \geq e^{\frac{L}{6}}\cdot \mu(C_L)
	\end{equation}
    On the other hand, the condition that $\int_X f d\mu = 0$ implies $\int_X S_N^-f d\mu=\int_X S_N^+f d\mu$.
    Using \autoref{trivial}, together with the trivial inequality $x\leq e^x$ for $x\geq 0$, we immediately get $\int_X S_N^+f d\mu \leq C$ for all $N\geq 1$.

    Together with \autoref{numbered}, this then implies $	\mu(C_L)\leq C'e^{-\frac{L}{6}}$.
	Using now the lower bound on $\mu(C_L)$ in terms of $B_L$ gives the desired $\mu(B_L)\leq C''e^{-\frac{L}{6}}$.
\end{proof}

Thanks to \autoref{exp0} and \autoref{exp+}, \autoref{prop:exponential_integrability_of_birkhoff_sums} applies to $f=\rho^u-\rho^s-h$, and so we conclude that for every $0<\gamma<\frac{1}{6}$ there is $C'>0$ such that
\begin{equation}\label{exp-}
\int_X e^{\gamma|S_N(\rho^u-\rho^s-h)|}d\mu \leq C',
\end{equation}
for all $N\geq 1$.
{In particular, thanks to the elementary inequality $\frac{\gamma^2}{2}|S_N(\rho^u-\rho^s-h)|^2\leq e^{\gamma|S_N(\rho^u-\rho^s-h)|},$ we see that} the hypotheses of \autoref{lem:finding_a_coboundary} are satisfied by $f=\rho^u-\rho^s - h$.
Indeed, the left hand side in \autoref{lem:finding_a_coboundary} is just $\norm{S_Nf}_{L^2(\mu)}^{1/2}$.
We thus obtain
\begin{corollary}
There is a function $\alpha\in L^2(\mu)$ such that
$$\rho^u-\rho^s-h=T^*\alpha-\alpha.$$
Hence, combining this with \autoref{eqn:the_sum_is_a_coboundary}, the functions
$$\alpha_u:=\frac{\alpha+\beta}{2},\quad \alpha_s:=\frac{\beta-\alpha}{2},$$
satisfy \autoref{eqn:rho_cohomologous_to_constant}.
\end{corollary}

\begin{remark}
By working in a suitable Orlicz space instead of $L^2(\mu)$, and applying the analog of \autoref{lem:finding_a_coboundary}, it is not hard to see that the coboundary $\alpha$ is in fact exponentially integrable, in the sense that
$e^{\gamma|\alpha|}\in L^1(\mu)$ for some $\gamma>0$. With more work, one can deduce the same integrability for $\beta$, and hence for $\alpha_u$ and $\alpha_s$. However, to show that $\alpha_u$ and $\alpha_s$ are in fact bounded, still requires the arguments that we used in \autoref{gen}.
\end{remark}


\bibliographystyle{sfilip_bibstyle}
\bibliography{K3_rigidity}

\end{document}